\documentclass{elsarticle}

\usepackage{mathdots}
\usepackage{latexsym}   
\usepackage{amssymb}    
\usepackage{amsmath}    
\usepackage{graphics}
\usepackage{enumerate}
\usepackage{amsfonts}
\usepackage{tikz}

\usepackage{multicol}
\usepackage{amsthm}     
\usetikzlibrary{automata,positioning,shapes,snakes}
\usetikzlibrary{decorations.markings}
\tikzstyle arrowstyle=[scale=1]
\tikzstyle directed=[postaction={decorate,decoration={markings,mark=at position .65 with {\arrow[arrowstyle]{stealth}}}}]

\newtheorem{theorem}{Theorem}[section]
\newtheorem{lemma}[theorem]{Lemma}

\newtheorem{corollary}[theorem]{Corollary}

\newtheorem{remark}[theorem]{Remark}
\newtheorem{example}[theorem]{Example}
\newtheorem{definition}[theorem]{Definition}

\usepackage{subfig}

\input diagxy

\def\O{{\cal O}}
\def\o{{\mathfrak o}}

\def\ZZ{{\mathbb{Z}}}

\def\CP{{\mathcal P}}
\def\CQ{{\mathcal Q}}
\def\CL{{\mathcal L}}
\def\CD{{\mathcal D}}

\def\CK{{\mathcal K}}

\newcommand \Jac {\mathop{\rm Jac}}

\newcommand \nin {\not \in}

\begin{document}
%
%

\begin{center}
{\bf Critical Groups of Graphs with Dihedral Actions II}

\vskip .25in

{\bf Darren B Glass}\\
{\small Department of Mathematics, Gettysburg College, Gettysburg PA 17325}\\
{\tt dglass@gettysburg.edu}\\
\end{center}

\noindent {\bf Abstract} 

In this paper we consider the critical group of finite connected graphs which admit harmonic actions by the dihedral group $D_n$, extending earlier work by the author and Criel Merino.  In particular, we show that the critical group of such a graph can be decomposed in terms of the critical groups of the quotients of the graph by certain subgroups of the automorphism group.  This is analogous to a theorem of Kani and Rosen which decomposes the Jacobians of algebraic curves with a $D_n$-action.

\section{Introduction}

This note  picks up where the author's previous article with Criel Merino \cite{GM} left off.  In particular, that article added to the growing literature (see, for example: \cite{BN1},\cite{Corry1},\cite{Primer}) exploring the analogy between the Jacobians of curves and the Jacobians of graphs, also known as critical groups.  Those papers, and others in the literature, prove theorems about Jacobians of graphs that are equivalent to theorems from algebraic geometry such as the Riemann-Roch Theorem and the Hurwitz bound on the size of the automorphism group.  Our article looked at a theorem of Kani and Rosen \cite{KR} that shows a relationship between the Jacobians of curves that admit certain group actions and the Jacobians of the quotients of that curve and explored whether their theorem carried over to the graph theoretic setting.  In particular, we showed how to decompose the Jacobian of a graph that had a harmonic action by the dihedral group $D_n$ in terms of the Jacobians of its quotients.  However, our results required the additional hypothesis that the $D_n$-orbits of the vertices each had precisely $n$ or $2n$ elements.  This hypothesis is very restrictive, and in this note, we remove that condition and prove the following theorem about Jacobians of graphs that admit a harmonic action of $D_n$ independent of the size of the orbits.

\begin{theorem}\label{T:Main}
Let $G$ be a graph admitting a harmonic action of the dihedral group $D_n$ generated by the involutions $\sigma_1$ and $\sigma_2$, and define an orbit of the vertices to be inertial if any element of the orbit is fixed by either $\sigma_1$ or $\sigma_2$.  For all primes $p \nmid 2n$, we have that the $p$-Sylow subgroups of the two groups \[\Jac(G/\sigma_1) \oplus \Jac(G/\sigma_2) \oplus \Jac(G/\sigma_1\sigma_2)\oplus \ZZ/n\ZZ\] \noindent and \[\Jac(G) \oplus \Jac(G/D_n)^2\oplus \left(\bigoplus_{\O \text{ inertial}}( \ZZ/\frac{n}{|\O|}\ZZ)\right)\]
\noindent are isomorphic.  If $p|2n$ then these $p$-Sylow subgroups have the same order but may not be isomorphic.
\end{theorem}

We believe this theorem is of  interest largely because it helps strengthen the connection between the study of Jacobians of graphs and Jacobians of curves.  Moreover, it is computationally useful because the quotient graphs $G/\sigma_1,G/\sigma_2$, and  $G/\sigma_1\sigma_2$  all have fewer vertices than $G$, and therefore computing their Jacobians directly from the Laplacian matrix will be faster.  When $n$ is odd we will get a further efficiency from the fact that $\sigma_1$ and $\sigma_2$ are conjugate elements and thus  $\Jac(G/\sigma_1) \cong \Jac(G/\sigma_2)$.  In particular, given that best algorithms for computing the critical group of a graph with $k$ vertices take somewhat less than $O(k^3)$ time \cite{KV}, this approach could speed up the computation by a factor of roughly $8$.

We note that when $n=2$, we are looking at the case where our graph admits a harmonic action of the Klein-Four group, and Theorem \ref{T:Main} simplifies as follows:

\begin{corollary}\label{T:Intro}
Let $K = \{id, \sigma_1,\sigma_2,\sigma_3\} \cong (\ZZ/2\ZZ)^2$  and let $G$ be a graph which admits a harmonic $K$-action so that there are exactly $\o$ points fixed by the entire group. Then for any prime $p \ne 2$, the $p$-part of the finite abelian group $\Jac(G) \oplus (\Jac(G/K))^2 $ is isomorphic to the $p$-part of the direct sum $\Jac(G/\sigma_1) \oplus \Jac(G/\sigma_2) \oplus \Jac(G/\sigma_3)$.  Moreover, if the $2$-part of $\Jac(G/\sigma_1) \oplus \Jac(G/\sigma_2) \oplus \Jac(G/\sigma_3)$ is of order $2^n$ then the $2$-part of $\Jac(G)\oplus (\Jac(G/K))^2 $ has order $2^{n-\o+1}$.
\end{corollary}

In the next section we introduce notation that we will use in our proof and discuss the general structure of our argument, which is very similar to the techniques used in \cite{GM}.  The following sections prove some technical results about sets of divisors on graphs that admit harmonic dihedral actions.  Section \ref{S:summary} combines these results in order to describe the Jacobian of our graph in terms of the Jacobians of its quotients.   A final section gives additional examples, including a proof of Corollary \ref{T:Intro}.

\section{Notation and Structure}\label{S:Notation}

Throughout this paper, we assume that $G$ is a graph that admits a harmonic action of the dihedral group $D_n$ generated by two involutions $\sigma_1$ and $\sigma_2$  and we set $\tau = \sigma_1\sigma_2$ so that $\tau$ is an element of order $n$.  We recall that the action is harmonic if anytime an element of the group fixes an edge it switches the two vertices that are endpoints of that edge.  Given a $D_n$-orbit $\O$ of vertices on the graph $G$, we wish to define the following two invariants: the type and the index.
\begin{definition}
For each orbit $\O$ of the vertices of $G$ under a harmonic $D_n$-action, we define it to be either Type I, II, or III as follows:
\begin{itemize}
\item A Type I orbit is an orbit $\O$ so that $\sigma_2$ fixes some element of $\O$.  We let $t_1$ be the number of orbits of Type I.
\item A Type II orbit is an orbit $\O$ so that $\sigma_2$ does NOT fix any element, but $\sigma_1$ does fix an element of $\O$. We let $t_2$ be the number of orbits of Type II.
\item A Type III orbit is be an orbit $\O$ so that neither $\sigma_1$ or $\sigma_2$ fix any elements of $\O$. We let $t_3$ be the number of orbits of Type III.
\end{itemize}
\end{definition}

Going along with the definition given in Theorem \ref{T:Main}, we will refer to Type I and Type II orbits jointly as `inertial', and we note that these are orbits on which the subgroup $\langle \tau \rangle$ acts transitively.  Similarly, we refer to Type III orbits as `non-inertial' and observe that the subgroup $\langle \tau \rangle$ splits these orbits into two suborbits of equal size.  In particular, if an orbit is inertial then the order of the orbit must be a divisor of $n$, and we define the index of $\O$ to be $k_\O= n/|\O|$.  On the other hand, the order of a non-inertial orbit will be twice a divisor of $n$, and we define the index to be $k_\O=2n/|\O|$.    When it is clear which orbit we are talking about, we will often drop the subscripts and just write $k$ for the index.  For a given graph $G$, we will find it useful to set $\kappa$ to be the least common multiple of the indices $k_\O$.

We note that if $n$ is odd then there are no Type II orbits, as $\sigma_1$ and $\sigma_2$ are conjugates of one another, and in particular $\tau^i\sigma_1  = \sigma_2\tau^i$ for some $i$.  Thus, if $\sigma_1(v)=v$ then $\sigma_2(\tau^i(v)) = \tau^i(v)$, meaning that $\sigma_2$ fixes an element in the same orbit as $v$.  A similar argument holds showing that there are no Type II orbits when $n$ is even and $\O$ is an orbit so that $n/k_\O$ is odd.  Moreover, without loss of generality we assume that if there are any Type II orbits then there are also Type I orbits; otherwise, we switch the labels of $\sigma_1$ and $\sigma_2$.

\begin{remark}\label{R:suborbit}
In what follows, it will be helpful to set the following notation.
\begin{itemize}
\item Any Type I orbit of index $k$ can be denoted by the set $\{z_i\}_{i=1}^{n/k}$ where $\sigma_1(z_i)=z_{n/k+1-i}$ and $ \sigma_2(z_i)=z_{n/k+2-i}$, where the subscripts should all be taken mod $n/k$.  In particular, it will have one fixed point under each involution  if $\frac{n}{k}$ is odd and two under $\sigma_2$ and none under $\sigma_1$ if $\frac{n}{k}$ is even.
\item Any Type II orbit of index $k$ can be denoted by the set $\{w_i\}_{i=1}^{n/k}$ where $\sigma_1(w_i)=w_{n/k-i}$ and $\sigma_2(w_i)=w_{n/k+1-i}$ with the subscripts again being taken mod $n/k$.  In particular, $\sigma_2$ will not fix any points while $\sigma_1$ will fix two points.
\item Any Type III orbit can be split into two suborbits $\{x_i\}_{i=1}^{n/k}, \{y_i\}_{i=1}^{n/k}$, where  $\sigma_1(x_i)=y_{n/k +1-i}$ and $\sigma_2(x_i) = y_{n/k +2-i}$, so the reflections take elements of one suborbit to the other, and the `rotations' generated by $\tau=\sigma_1\sigma_2$ fix each of the suborbits.
\end{itemize}
We illustrate this notation in Figure \ref{F:orbits}, which shows graphs admitting $D_3$ and $D_4$ actions.  In each case, $\sigma_1$ denotes reflection in the vertical axis and $\sigma_2$ is reflection in the diagonal line indicated.  The graphs each have one orbit of each Type, all of index one and labeled with the above notation.

\begin{figure}[!htbp]
\centering
\subfloat[Graph with $D_3$ action]{\begin{tikzpicture}
  [scale=1,auto=left,every node/.style={state,minimum size=.6cm,fill=blue!20}]
  \node(w1) at (.86,.5) {$w_1$};
  \node(w2) at (0,-1){$w_2$};
  \node(w3) at (-.86,.5){$w_3$};
  \node(x1) at (1.1,2.2) {$x_1$};
  \node(y1) at (2.5,0){$y_1$};
  \node(x2) at (1,-2.2){$x_2$};
  \node(y2) at (-1,-2.2) {$y_2$};
  \node(x3) at (-2.5,0){$x_3$};
  \node(y3) at (-1.1,2.2){$y_3$};
  \foreach \from/\to in {w1/w2,w2/w3,w3/w1,w1/x1,w1/y1,x1/y3,w2/x2,w2/y2,x2/y1,w3/x3,w3/y3,x3/y2}
    \draw (\from) -- (\to);
      \draw[dotted] (0,2.7) -- (0,-2.7) ;
    \draw[dotted] (-2.5,-1.5) -- (2.5,1.5);
\end{tikzpicture}} \qquad
\subfloat[Graph with $D_4$ action]{\begin{tikzpicture}
  [scale=1.2,auto=left,every node/.style={state,minimum size=.6cm,fill=blue!20}]
  \node(z1) at (1,1) {$z_1$};
  \node(z2) at (1,-1){$z_2$};
  \node(z3) at (-1,-1){$z_3$};
  \node(z4) at (-1,1) {$z_4$};
  \node(w1) at (1,0){$w_1$};
  \node(w2) at (0,-1){$w_2$};
  \node(w3) at (-1,0) {$w_3$};
  \node(w4) at (0,1){$w_4$};
  \node(x1) at (.5,2) {$x_1$};
  \node(x2) at (2,-.5){$x_2$};
  \node(x3) at (-.5,-2){$x_3$};
  \node(x4) at (-2,.5) {$x_4$};
  \node(y1) at (2,.5){$y_1$};
  \node(y2) at (.5,-2){$y_2$};
  \node(y3) at (-2,-.5) {$y_3$};
  \node(y4) at (-.5,2){$y_4$};
  \foreach \from/\to in {x1/y1,y1/x2,x2/y2,y2/x3,x3/y3,y3/x4,x4/y4,y4/x1, z1/w1,w1/z2,z2/w2,w2/z3,z3/w3,w3/z4,z4/w4,w4/z1,w4/x1,w3/x4,w2/x3,w1/x2,w1/y1,w2/y2,w3/y3,w4/y4}
    \draw (\from) -- (\to);
      \draw[dotted] (0,2.5) -- (0,-2.5) ;
    \draw[dotted] (-1.8,-1.8) -- (1.8,1.8);
\end{tikzpicture}}
\caption{Graphs with harmonic actions}
\label{F:orbits}
\end{figure}
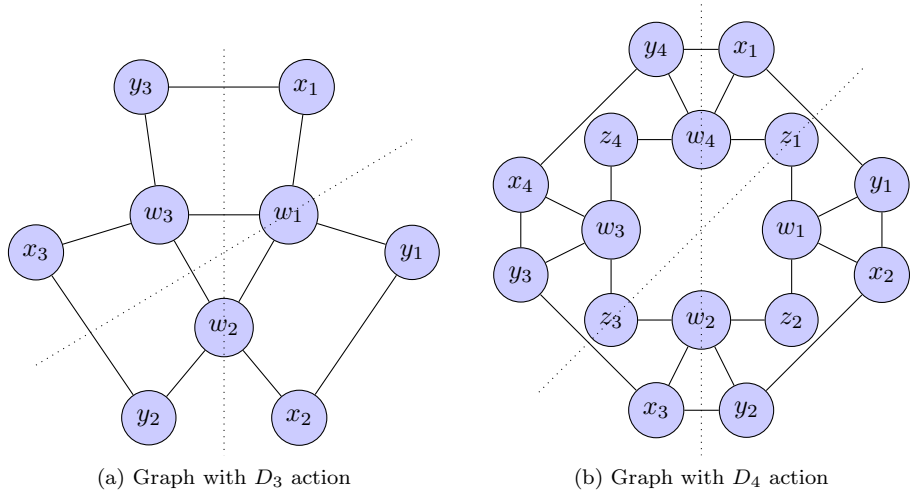

\end{remark}

\begin{example}\label{E:K44}
Throughout this note, we will consider the running example of the complete bipartite graph $K_{4,4}$, as pictured in Figure \ref{F:K44}.  While the Jacobian of this graph is easy to compute from first principles, we believe it will be useful to keep in mind when working through our results.  With vertices labelled as in the picture, we define involutions on the vertices by letting $\sigma_1 = (z_1 \, z_4) (z_2 \, z_3) (x \, y)$ and $\sigma_2 = (z_2 \, z_4)(w_1 \, w_2) (x \, y)$. One can compute that the product $\tau=\sigma_2\sigma_1 = (z_1 \, z_2 \, z_3 \, z_4) (w_1 \,w_2)$ has order four, and therefore the group generated by $\langle \sigma_1,\sigma_2 \rangle$ is isomorphic to the dihedral group $D_4$.

\begin{figure}[!htbp]
\begin{center}
\begin{tikzpicture}
  [scale=1,auto=left,every node/.style={state,minimum size=.8cm,fill=blue!20}]
  \node(x1) at (-1,1) {$z_4$};
  \node(x2) at (-1,2){$z_3$};
  \node(x3) at (-1,3){$z_2$};
  \node(x4) at (-1,4) {$z_1$};
  \node(y1) at (1,1){$w_2$};
  \node(y2) at (1,2){$w_1$};
  \node(z1) at (1,3) {$y$};
  \node(z2) at (1,4){$x$};
  \foreach \from/\to in {x1/y1,x1/y2,x1/z1,x1/z2,x2/y1,x2/y2,x2/z1,x2/z2,x3/y1,x3/y2,x3/z1,x3/z2,x4/y1,x4/y2,x4/z1,x4/z2}
    \draw (\from) -- (\to);
\end{tikzpicture}
\caption{$D_4$ action on $K_{4,4}$}
\label{F:K44}
\end{center}
\end{figure}
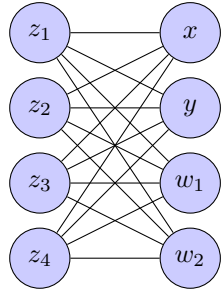

It is straightforward to check that this group action is harmonic. The vertices of $K_{4,4}$ have three orbits under this group action:
\begin{itemize}
\item $\{z_1,z_2,z_3,z_4\}$ is an orbit of Type I and index $k=1$.
\item $\{w_1,w_2\}$ is an orbit of Type II and index $k = 2$.
\item $\{x,y\}$ is an orbit of Type III and index $k= 4$.
\end{itemize}
Moreover, we see that $\kappa = lcm(1,2,4)=4$.
\end{example}

One consequence of a morphism $\phi:G \rightarrow H$ between graphs being harmonic is that for any vertex $w \in V(H)$ and $v \in V(G)$ so that $\phi(v)=w$ we have that all edges of $H$ adjacent to $w$ have the same number of preimages in $G$ that are adjacent to $v$; following Baker and Norine in \cite{BN2} we call this the horizontal multiplicity of $\phi$ at $w$ and denote it by $m_\phi(w)$.  Recalling that a divisor on a graph can be thought of as a function $\delta$ from the vertices of the graph to $\ZZ$, we define the pullback map from the set of divisors on $H$ to the set of divisors on $G$ by setting $\displaystyle \phi^*(\delta)=\sum_{w \in V(H)} \sum_{\substack{v \in V(G) \\ \phi(v)=w}} m_\phi(w) (\delta(w) )(v)$.

\begin{definition}\label{D:sets} Let $G$ be a graph that admits a harmonic $D_n$-action.  We define the following sets of divisors on the graph $G$:
\begin{itemize}
\item $\CD$ is the set of divisors $\delta$ so that $\sum_{v \in G} \delta(v)=0$.
\item $\CP_0$ is the set of divisors $\delta$ that can be viewed as pullbacks of divisors of degree zero on the quotient graph $\hat{G} = G/D_n $.  In particular, it follows from \cite[Lemma 3.2]{GM} that $\CP_0$ is the set of divisors of total degree zero whose values are constant on each orbit $\O$ and whose value on the vertices of $\O$ are all multiples of $2n/|\O|$.
\item For $i=1,2$ we set $\CP_i$  to be the set of divisors that can be viewed as pullbacks of divisors of degree zero on the quotient graph $H_i = G/\langle \sigma_i \rangle $.  In particular, $\CP_i$ will consist of divisors of degree zero so that $\delta(v)=\delta(\sigma_i(v))$ for all $v$ and $\delta(v)$ is even if $\sigma_i(v)=v$.
\item  $\CP_3$  is the set of divisors $\delta$ that can be viewed as pullbacks of divisors of degree zero on the quotient graph $H_3 = G/\langle \tau \rangle $ where $\tau=\sigma_1\sigma_2$. We have that $\delta \in \CP_3$ if and only if $\delta$ is of total degree zero, $\delta(v)$ is constant on all orbits of Type I and II  and constant on each of the $x_i$ and $y_i$ suborbits of any orbit of Type III and, moreover, that the value of $\delta(v)$ is a multiple of the index of the orbit containing $v$ for all vertices.
\item We set $\CP = \CP_1+\CP_2+\CP_3$.
\end{itemize}
\end{definition}

\begin{example}\label{E:PiK}
We illustrate these definitions by computing these sets for the bipartite graph with the $D_4$ action described in Example \ref{E:K44}.  Figure \ref{F:PiK} illustrates the symmetry conditions for the sets $\CP_i$; note that a divisor will be in the set $\CP_i$ exactly when it takes this form (where all of the letters represent integer values) and has total degree equal to zero.  Understanding the set $\CP = \CP_1+\CP_2+\CP_3$ is in general more subtle and will be the main topic of Section \ref{S:sums}.

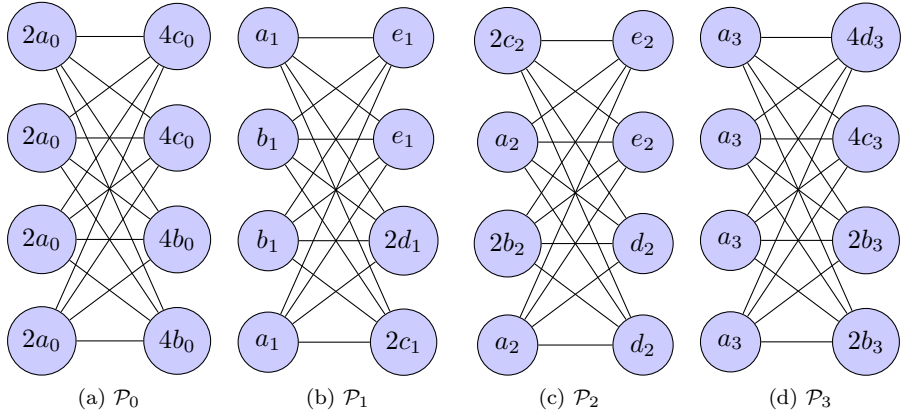
\begin{figure}[!htbp]
\centering
\subfloat[$\CP_0$]{\begin{tikzpicture}
  [scale=.9,auto=left,every node/.style={state,minimum size=.8cm,fill=blue!20}]
  \node(x1) at (-1,1) {$2a_0$};
  \node(x2) at (-1,2.5){$2a_0$};
  \node(x3) at (-1,4){$2a_0$};
  \node(x4) at (-1,5.5) {$2a_0$};
  \node(y1) at (1,1){$4b_0$};
  \node(y2) at (1,2.5){$4b_0$};
  \node(z1) at (1,4) {$4c_0$};
  \node(z2) at (1,5.5){$4c_0$};
  \foreach \from/\to in {x1/y1,x1/y2,x1/z1,x1/z2,x2/y1,x2/y2,x2/z1,x2/z2,x3/y1,x3/y2,x3/z1,x3/z2,x4/y1,x4/y2,x4/z1,x4/z2}
    \draw (\from) -- (\to);
\end{tikzpicture}} \quad
\subfloat[$\CP_1$]{\begin{tikzpicture}
  [scale=.9,auto=left,every node/.style={state,minimum size=.8cm,fill=blue!20}]
  \node(x1) at (-1,1) {$a_1$};
  \node(x2) at (-1,2.5){$b_1$};
  \node(x3) at (-1,4){$b_1$};
  \node(x4) at (-1,5.5) {$a_1$};
  \node(y1) at (1,1){$2c_1$};
  \node(y2) at (1,2.5){$2d_1$};
  \node(z1) at (1,4) {$e_1$};
  \node(z2) at (1,5.5){$e_1$};
  \foreach \from/\to in {x1/y1,x1/y2,x1/z1,x1/z2,x2/y1,x2/y2,x2/z1,x2/z2,x3/y1,x3/y2,x3/z1,x3/z2,x4/y1,x4/y2,x4/z1,x4/z2}
    \draw (\from) -- (\to);
\end{tikzpicture}}
 \quad
\subfloat[$\CP_2$]{
\begin{tikzpicture}
  [scale=.9,auto=left,every node/.style={state,minimum size=.8cm,fill=blue!20}]
  \node(x1) at (-1,1) {$a_2$};
  \node(x2) at (-1,2.5){$2b_2$};
  \node(x3) at (-1,4){$a_2$};
  \node(x4) at (-1,5.5) {$2c_2$};
  \node(y1) at (1,1){$d_2$};
  \node(y2) at (1,2.5){$d_2$};
  \node(z1) at (1,4) {$e_2$};
  \node(z2) at (1,5.5){$e_2$};
  \foreach \from/\to in {x1/y1,x1/y2,x1/z1,x1/z2,x2/y1,x2/y2,x2/z1,x2/z2,x3/y1,x3/y2,x3/z1,x3/z2,x4/y1,x4/y2,x4/z1,x4/z2}
    \draw (\from) -- (\to);
\end{tikzpicture}} \quad
\subfloat[$\CP_3$]{\begin{tikzpicture}
  [scale=.9,auto=left,every node/.style={state,minimum size=.8cm,fill=blue!20}]
  \node(x1) at (-1,1) {$a_3$};
  \node(x2) at (-1,2.5){$a_3$};
  \node(x3) at (-1,4){$a_3$};
  \node(x4) at (-1,5.5) {$a_3$};
  \node(y1) at (1,1){$2b_3$};
  \node(y2) at (1,2.5){$2b_3$};
  \node(z1) at (1,4) {$4c_3$};
  \node(z2) at (1,5.5){$4d_3$};
  \foreach \from/\to in {x1/y1,x1/y2,x1/z1,x1/z2,x2/y1,x2/y2,x2/z1,x2/z2,x3/y1,x3/y2,x3/z1,x3/z2,x4/y1,x4/y2,x4/z1,x4/z2}
    \draw (\from) -- (\to);
\end{tikzpicture}}

\caption{Generic elements of the sets $\CP_i$ for the $D_4$ action on $K_{4,4}$}
\label{F:PiK}
\end{figure}

\end{example}

\begin{definition}
In addition to understanding the sets of divisors, it is important to understand the operation of `firing vertices' that defines the Jacobian of a graph.  More precisely:
\begin{itemize}
\item For each vertex $v$ we define the divisor $\ell_v$ to be the divisor corresponding to `firing' at $v$.  In particular, $\ell_v(v)=-deg(v)$ while for all $w \ne v$ we set $\ell_v(w)$ equal to the number of edges between $v$ and $w$.  We define $\CL$ to be the set of divisors that can be written as integer linear combinations of the $\ell_v$, and note that $\Jac(G)$ is defined to be the quotient group $\CD/\CL$.
\item $\CL'$ is defined to be the set which is generated by those divisors on $G$ that are pullbacks of the divisors corresponding to firing vertices on $H_1, H_2$ and $H_3$.  We will explore this set more concretely in Section \ref{S:LL'}.
\end{itemize}
\end{definition}

In \cite{GM}, we show that we have the following inclusion of these groups

\[ \bfig \morphism(-100,100)/{^(->}/<400,0>[\CL'`\CP \cap \CL;]
\morphism(300,100)/{_(->}/<300,-300>[\CP \cap \CL`\CP;]
\morphism(300,100)/{^(->}/<300,300>[\CP \cap \CL`\CL;]
\morphism(600,-200)/{_(->}/<300,300>[\CP`\CP+\CL;]
\morphism(600,400)/{^(->}/<300,-300>[\CL`\CP+\CL;]
\morphism(900,100)/{^(->}/<400,0>[\CP+\CL`\CD;]
\efig \]

\noindent which allows us to use the isomorphism theorems in order to prove the following result:

\begin{theorem} \label{T:Exact}
We have the following exact sequences between the quotients of groups:
\[ 1 \rightarrow (\CP \cap \CL)/\CL' \rightarrow \CL/\CL' {\rightarrow} \CD/\CP \rightarrow \CD/(\CP + \CL) \rightarrow 1\]
\[ 1 \rightarrow (\CP \cap \CL)/\CL' \rightarrow \CP/\CL' {\rightarrow} \Jac(G) \rightarrow \CD/(\CP + \CL) \rightarrow 1\]
\[ 1 \rightarrow \CK \rightarrow \Jac(H_1) \oplus \Jac(H_2) \oplus \Jac(H_3) \rightarrow \CP/\CL' \rightarrow 1\]
\end{theorem}

Our proof of this theorem is independent of the size of the orbits and therefore applies to the more general setting of this note.  In that paper we then use the additional hypotheses to prove specific results about the groups in these exact sequence.  We will follow a similar approach here, and subsequent sections will give explicit formulations for the finite abelian groups $\CD/\CP$ (Theorem \ref{T:DP}), $\CK$ (Theorem \ref{T:K}), and $\CL/\CL'$ (Theorem \ref{T:LL'}).  Along with the above exact sequences, these results determine the relationship between $\Jac(G)$ and $\Jac(H_1) \oplus \Jac(H_2) \oplus \Jac(H_3)$, allowing us to prove Theorem \ref{T:Main}.

\section{Sums of Pullbacks}\label{S:sums}

Our goal in this section is to understand the group $\CD/\CP$, and Theorem \ref{T:DP} will give a precise formula for this group.  To prove this theorem, we give a precise characterization of the divisors in $\CP$. We begin by considering the sum $\CP_1 + \CP_2$,  for which we need to define some auxiliary functions.

\begin{definition}
For each orbit $\O$ of vertices of $G$ under the $D_n$ action we define a function $F_\O: \{$divisors on $G\} \rightarrow \ZZ$ as follows:
\begin{itemize}
\item Let $\O = \{z_i\}$ denote a Type I  orbit of vertices of $G$ using the notation of Remark \ref{R:suborbit}.  Then $ F_\O(\delta) =  \sum_{i=1}^{n/k} 2i \delta(z_i)$.
\item If $\O=\{w_i\}$ is a Type II orbit then  $ F_\O(\delta) =  \sum_{i=1}^{n/k} (2i+1) \delta(w_i)$.
\item If $\O=\{x_i,y_i\}$ is a Type III orbit then  $F_\O(\delta) =  \sum_{i=1}^{n/k} 2i (\delta(x_i)+ \delta(y_i))$.
\end{itemize}
Note that we define $F_\O(\delta)$ as a function on the set of all divisors $\delta$ on $G$, but its value depends only on the value of the divisor at vertices in the orbit $\O$.
\end{definition}

We note the similarities between the function $F_\O$ and the function $\tau$ defined in \cite[Defn 2]{Cairns}.  The individual functions $F_\O$ can vary quite a bit, but it turns out that adding these together captures important global information about the divisor $\delta$, and we define the function $F(\delta) = \sum_\O F_\O(\delta)$.

\begin{example}
Returning to our running example, we consider a divisor $\delta_1$ on $K_{4,4}$ that is in the set $\CP_1$ as illustrated in Example \ref{E:PiK}.  Keeping in mind that we must have $a_1+b_1+c_1+d_1+e_1=0$, we compute that
\begin{eqnarray}
F(\delta_1) &=& \sum_{i=1}^{4} 2i \delta_1(z_i) + \sum_{i=1}^{2} (2i+1) \delta_1(w_i)+ 2(\delta_1(x)+ \delta_1(y))\\
&=& 2a_1+4b_1+6b_1+8a_1+3e_1+5e_1+4c_1+4d_1 \\
&=&10a_1+10b_1+8e_1+4c_1+4d_1 \\
&=&6a_1+6b_1+4e_1
\end{eqnarray}
which we note is an even number. Similarly, if $\delta_2 \in \CP_2$ we can compute that $F(\delta_2) = 8a_2+8b_2+4d_2$, which will also be even. Because the definition of $F_\O(\delta)$ is linear in $\delta$, it follows that any divisor $ \delta \in \CP_1 + \CP_2$ will also have that $F(\delta)$ is even.
\end{example}

More generally, we can prove the following result:

\begin{lemma}\label{L:global}
Let $G$ be a graph with a $D_n$ action and let $\kappa$ be the least common multiple of the indices $k_\O$ of the orbits of this action.  For every divisor $\delta \in \CP_1$  we have that $F(\delta)$ is a multiple of $2n/\kappa$.  The same result holds for divisors $\delta \in \CP_2$.
\end{lemma}

\begin{proof}
Let  $\delta \in \CP_1$  and let $\O$ be a Type I orbit, so that  $\delta(z_i) = \delta(z_{n/k_\O+1-i})$.  We then compute:
\begin{eqnarray*}
F_\O(\delta) &=&  \sum_{i=1}^{n/k_\O}2i \delta(z_i)\\
&=& \sum_{i=1}^{n/k_\O} \left(i \delta(z_i)+i \delta(z_{n/k_\O+1-i})\right)\\
&=& \sum_{i=1}^{n/k_\O} i \delta(z_i)+ \sum_{i=1}^{n/k_\O} i \delta(z_{n/k_\O+1-i})\\
&=& \sum_{i=1}^{n/k_\O} i \delta(z_i)+ \sum_{j=1}^{n/k_\O} (n/k_\O+1-j) \delta(z_j)\\
&=& \sum_{i=1}^{n/k_\O} \left(i \delta(z_i)+(n/k_\O+1-i) \delta(z_i)\right)\\
&=& \left(\frac{n}{k_\O}+1\right)\sum_{v \in \O}  \delta(v)
\end{eqnarray*}

\noindent where the equality at line three comes from setting $j=n/k_\O+1-i$.

Similarly, one can use the symmetry properties of Type II and Type III orbits to see that the same conclusion holds for those orbits as well.  Using these facts as well as the fact that $\delta$ must have total degree equal to zero, we then compute that
\begin{eqnarray*}
F(\delta)&=& \sum_\O \left((\frac{n}{k_\O} +1) \sum_{v \in \O}  \delta(v)\right)\\
&=&\sum_\O \left(\frac{n}{k_\O} \sum_{v \in \O}  \delta(v)\right) + \sum_{v \in G} \delta(v)\\
&=&\sum_\O \left(\frac{n}{k_\O} \sum_{v \in \O}  \delta(v)\right)\\
&=&\frac{n}{\kappa} \sum_\O \left(\frac{\kappa}{k_\O} \sum_{v \in \O}  \delta(v)\right)\\
\end{eqnarray*}

\noindent which is clearly a multiple of $n/\kappa$.  Moreover, for symmetry reasons we know that if $\delta \in \CP_1$ then for each orbit $\O$ we have that $\sum_{v \in \O}  \delta(v)$ will be even, which in turn implies that $\frac{\kappa}{k_\O} \sum_{v \in \O}  \delta(v)$ is even.   The desired result follows for divisors in $\CP_1$.  The argument is similar for divisors in $\CP_2$.
\end{proof}

We are now able to prove a result about the divisors that can be written as the sum of divisors in $\CP_1$ and $\CP_2$.

\begin{theorem}\label{T:P12}
 A divisor $\delta \in \CD$ will be in the set $\CP_1+\CP_2$ if and only if $\delta$ satisfies the following conditions:
\begin{enumerate}[(1)]
\item For each orbit $\O$, we have that $\sum_{v \in \O} \delta(v)$ is even.
\item If $\O = \{x_i\} \cup \{y_i\}$ is an orbit of Type III  then $\sum \delta(x_i)=\sum\delta(y_i)$.
\item $F(\delta)$ is a multiple of $2n/\kappa$, where $\kappa$ is defined as above.
\end{enumerate}
\end{theorem}

\begin{proof}
It is straightforward to check that that the first two conditions holds for all elements in $\CP_1$ and $\CP_2$.   The fact that the third condition holds for elements of $\CP_1$ and $\CP_2$ is the content of Lemma \ref{L:global}.  Moreover, these conditions are all linear and will therefore hold for all elements of $\CP_1+\CP_2$.  It follows that the conditions are necessary for a divisor to be in $\CP_1+\CP_2$.   To check that these conditions are also sufficient, let $\delta$ be a divisor satisfying all three of them.  We wish to define divisors $\delta_1 \in \CP_1$ and $\delta_2 \in \CP_2$ so that $\delta = \delta_1 + \delta_2$.

Using the notation above, we first define auxiliary divisors $\hat{\delta}_1$ and $\hat{\delta}_2$ on orbits of Type III as follows:

\begin{eqnarray*}
\hat{\delta}_1(x_i) &=& (\delta(x_1) + \dotsm + \delta(x_i)) - (\delta(y_{n/k})+\dotsm + \delta(y_{n/k+2-i}))\\
\hat{\delta}_1(y_i) &=& \hat{\delta}_1(x_{n/k+1-i}) \\
\hat{\delta}_2(x_i) &=& (\delta(y_{n/k})+\dotsm + \delta(y_{n/k+2-i}) - (\delta(x_1) + \dotsm + \delta(x_{i-1}))\\
\hat{\delta}_2(y_i) &=& \hat{\delta}_2(x_{n/k+2-i})
\end{eqnarray*}

We note that we need Condition (2) of the Lemma to hold in order for these terms to be well-defined.  It is clear by definition that $\hat{\delta}_1$ and $\hat{\delta}_2$ satisfy the required symmetry properties.  One can easily see that $\hat{\delta}_1(x_i)+\hat{\delta}_2(x_i)=\delta(x_i)$ and we compute:
\begin{eqnarray*}
\hat{\delta}_1(y_i)+\hat{\delta}_2(y_i) &=&\hat{\delta}_1(x_{n/k+1-i}) +\hat{\delta}_2(x_{n/k+2-i})\\
&=& \left(\sum_{j=1}^{n/k+1-i}\delta(x_j)- \sum_{j=i+1}^{n/k}\delta(y_j)\right)+ \left(\sum_{j=i}^{n/k}\delta(y_j)- \sum_{j=1}^{n/k+1-i}\delta(x_j)\right)\\
&=& \delta(y_i)
\end{eqnarray*}

For  inertial orbits, we define $\hat{\delta}_1$ and $\hat{\delta}_2$ as follows:

\begin{eqnarray*}
\hat{\delta}_1(z_i) &=& (\delta(z_1)+ \dotsm +\delta(z_i)) - (\delta(z_{n/k})+ \dotsm +\delta(z_{n/k+2-i}))\\
\hat{\delta}_1(w_i)&=& (\delta(w_1)+ \dotsm +\delta(w_i)) - (\delta(w_{n/k-1})+ \dotsm +\delta(w_{n/k+1-i}))\\
\hat{\delta}_2(z_i) &=& (\delta(z_{n/k})+ \dotsm +\delta(z_{n/k+2-i}))-(\delta(z_1)+ \dotsm +\delta(z_{i-1})) \\
\hat{\delta}_2(w_i) &=&(\delta(w_{n/k-1})+ \dotsm +\delta(w_{n/k+1-i})) - (\delta(w_1)+ \dotsm +\delta(w_{i-1}))\\
\end{eqnarray*}

The reader can check that  each $\hat{\delta}_i$ satisfies the required parity and symmetry conditions on all orbits in order to be in $\CP_i$, as well as the fact that the value of $\hat{\delta}_1+\hat{\delta}_2$ on these vertices agrees with the value of $\delta$.  Unfortunately, it may not be the case that the $\hat{\delta}_i$ have total degree zero, and therefore we must adjust them accordingly.  In particular, we look at the sum of the values of $\hat{\delta}_1$ on specific orbits and see the following:

\begin{alignat*}{4}
\O \text{ Type I:} & \sum_{v \in \O} \hat{\delta}_1(v)  &=& \sum_{i=1}^{n/k_\O} \left(\frac{n}{k_\O}+2-2i\right) \delta(z_i)  \\
\O \text{ Type II:} & \sum_{v \in \O} \hat{\delta}_1(v) &=& \sum_{i=1}^{n/k_\O} \left(\frac{n}{k_\O}+1-2i\right) \delta(w_i) \\
\O \text{ Type III:} & \sum_{v \in \O} \hat{\delta}_1(v) &=&  \sum_{i=1}^{n/k_\O} \left(\frac{n}{k_\O}+2-2i\right) \left(\delta(x_i) + \delta(y_i)\right) \\
\end{alignat*}

Thus, for each orbit $\O$ we have $\sum_{v \in \O} \hat{\delta}_1(v) =  (\frac{n}{k_\O}+2) \sum_{v \in \O}\delta(v) - F_\O(\delta)$, and it therefore follows that
\[\sum_{v \in G} \hat{\delta}_1(v) = 2 \sum_{v \in G} \delta(v) + \sum_\O\left(\frac{n}{k_\O} \sum_{v \in \O}\delta(v) - F_\O(\delta)\right)\]

Using Conditions (1) and (3) in the statement of the lemma, we can conclude that this is a multiple of $2n/\kappa$. Moreover, we note that $\sum_{v \in G} \delta(v) = 0$ so in particular $\sum_{v \in G} \hat{\delta}_1(v) = - \sum_{v \in G} \hat{\delta}_2(v)$.  Recall that, because $\sum_{v \in G} \hat{\delta}_1(v)$ is a multiple of $2n/\kappa$, there exists a divisor $\tilde{\delta} \in \CQ$  with total degree equal to $\sum_{v \in G} \hat{\delta}_1(v)$. Choose one such divisor and set $\delta_1 = \hat{\delta}_1-\tilde{\delta}$ and $\delta_2 = \hat{\delta}_2+\tilde{\delta}$.  One can easily verify that $\delta_1 \in \CP_1$ and $\delta_2 \in \CP_2$, and $\delta_1+\delta_2=\delta$, completing the proof.
\end{proof}

We note that in the case where the graph has any points that are fixed by all group elements there is an orbit of index $n$ and in particular, $\kappa=n$.  In this case, the third condition in Theorem \ref{T:P12} becomes vacuously true.  At the other extreme,  if all orbits have index $1$ then this result is equivalent to \cite[Thm 3.5]{GM}.

\begin{example}
Returning to Example \ref{E:K44}, we have that $n=\kappa=4$ and therefore we again see that the third condition in Theorem \ref{T:P12} follows immediately from the others.  In particular, applying this result to our example, we see that a divisor $\delta$ will be in $\CP_1+\CP_2$ if and only if its total degree is equal to zero and the following conditions all hold:
\begin{itemize}
\item $\delta(w_1)+\delta(w_2)+\delta(w_3)+\delta(w_4)$ is even,
\item $\delta(z_1)$ and $\delta(z_2)$ have the same parity, and
\item $\delta(x)=\delta(y)$
\end{itemize}
\end{example}

The functions $F_\O(\delta)$ are also well-behaved if $\delta \in \CP_3$, and the following lemma is a straightforward calculation from the definitions:

\begin{lemma}\label{L:F3}
Let $\delta$ be a divisor in $\CP_3$.
\begin{itemize}
\item On any orbit $\O$ of Type I, the divisor $\delta$ takes on a constant value of $a_\O k_\O$, and we have $F_\O(\delta) = na_\O (n/k_\O+1)$
\item On any orbit $\O$ of Type II, the divisor $\delta$ takes on a constant value of $a_\O k_\O$, and we have $F_\O(\delta) = na_\O (n/k_\O+2)$
\item On any orbit $\O$ of Type III, the divisor $\delta$ takes on a constant value of $a_\O k_\O$ (resp. $b_\O k_\O$) on the $x_i$ (resp. $y_i$) suborbit, and we can compute that $F_\O(\delta) = n(a_\O+b_\O) (n/k_\O+1)$
\end{itemize}
\end{lemma}

We now wish to categorize the divisors that are elements of $\CP= \CP_1+\CP_2+\CP_3$.  Our categorization will break into different cases depending on the parity of $n$ and $\kappa$, and we cover three different cases in the next three theorems.  We will give a complete proof of the case where $n$ is odd in Theorem \ref{T:Podd}, but for the subsequent theorems we choose to highlight where the proof is different and leave the remaining details to the reader.

\begin{theorem}\label{T:Podd}
Assume $n$ is odd.  Then a divisor $\delta \in \CP_1+\CP_2+\CP_3$ if and only if $\delta$ satisfies the following conditions:
\begin{itemize}
\item $\sum_{v} \delta(v)=0$
\item For all Type III orbits $\O = \{x_i\} \cup \{y_i\}$, we have that  $\sum \delta(x_i)\equiv \sum\delta(y_i)$ (mod $n$).
\item $F(\delta)$ is a multiple of $n/\kappa$.
\end{itemize}
\end{theorem}

\begin{proof}
The conditions in the hypothesis of this Theorem are less restrictive than those in Theorem \ref{T:P12} and therefore any divisor that is in $\CP_1+\CP_2$ will automatically satisfy them.  We wish to show that if $\delta_3 \in \CP_3$ that it will also satisfy the conditions in the statement of the Theorem.  By definition, all divisors in $\CP_3$ have total degree equal to zero, ensuring that $\delta_3$ satisfies the first condition.  From the remarks in Definition \ref{D:sets} we know that if $\delta_3 \in \CP_3$ then it is constant on the $x_i$ and $y_i$ suborbits of all Type III orbits.  In particular, for each Type III orbit we have that $\sum \delta_3(x_i) \equiv \sum \delta_3(y_i) \equiv 0$ mod $n$, so it satisfies the second condition as well.

Recall that the hypothesis that $n$ is odd implies that all orbits are either of Type I or Type III.  Lemma \ref{L:F3} therefore implies that for each $\O$ we have  that $F_\O(\delta_3)$ is a multiple of $n/k_\O$, so will be a multiple of $n/\kappa$.  This implies that $F(\delta)$ is a multiple of $n/\kappa$.  Therefore, we see that any divisor in $\CP_3$ will satisfy the conditions in the statement of the theorem, and by linearity it is clear that any divisor in $\CP=\CP_1+\CP_2+\CP_3$ will as well.

Conversely, assume that $\delta$ satisfies the conditions in the statement of the theorem.  We wish to define a divisor $\delta_3 \in \CP_3$ so that $\delta-\delta_3 \in \CP_1+\CP_2$.  To do this, we first note that for each orbit of Type III we have by hypothesis that $\sum \delta(x_i)\equiv\sum\delta(y_i)$ (mod $n$).  This implies that  $(\sum \delta(x_i)- \sum\delta(y_i))/n$ is an integer, and we set this integer equal to $a_\O$.  Moreover, let $\alpha = \sum_\O a_\O$.  We now consider separately the case where $n$ does and does not have an orbit of Type I.
\vskip .1in
\noindent {\bf Case 1:} At least one orbit of Type I.

On each orbit $\O$ of Type III, let us define $\delta_3$ on the suborbits by setting $\delta_3(x_i) =k_\O a_\O$ and $\delta_3(y_i) = 0$ for all $i$.  By hypothesis, there exists at least one Type I orbit so choose one of them and designate it $\Omega$.  On all other Type I orbits, we define $\delta_3(v) = 0$ if $\sum_{v \in \O} \delta(v)$ is even and $\delta_3(v) = k_\O$ if $\sum_{v \in \O} \delta(v)$ is odd.  For vertices in $\Omega$, we define $\delta_3$ so that the total degree of the divisor is zero; in particular, for all $\omega \in \Omega$ we set $\delta_3(\omega) = -k_\Omega(\alpha+\hat{t}_1)$, where $\hat{t}_1$ is the number of Type I orbits so that $\sum_{v \in \O} \delta(v)$ is odd.  It is clear that $\delta_3$ satisfies the necessary symmetry properties to be in $\CP_3$ and also that $k_\O|\delta(v)$ for each $v \in \O$.  Therefore, $\delta_3 \in \CP_3$ and we set $\hat{\delta}=\delta-\delta_3$.  We wish to show that $\hat{\delta} \in \CP_1+\CP_2$ by showing that it satisfies the hypotheses of Theorem \ref{T:P12}.

To begin, we compute that for any orbit of Type III:
\begin{eqnarray*}
\sum \hat{\delta}(x_i) &=& \sum \delta(x_i) - \sum \delta_3(x_i)\\
&=& \sum \delta(x_i) - \frac{n}{k_\O}k_\O a_\O \\
&=& \sum \delta(x_i) - n \cdot \frac{\sum \delta(x_i)- \sum\delta(y_i)}{n} \\
&=&\sum \delta(y_i)\\
&=&\sum \hat{\delta}(y_i)
\end{eqnarray*}

This shows that Condition (2) is satisfied, and also implies that $\sum_{v \in \O} \hat{\delta}(v)$ is even on orbits of Type III.  For orbits of Type I, one can check that by construction we have that $\sum_{v \in \O} \delta_3(v) \equiv \sum_{v \in \O} \delta(v)$ (mod $2$) and therefore $\sum_{v \in \O} \hat{\delta}(v)$ is even on orbits of Type I as well.  Therefore, the first two conditions of Theorem \ref{T:P12} apply to $\hat{\delta}$.

To see that the third holds as well, we note that the fact that there are no orbits of Type II implies that for any divisor $\beta$, it must be the case that $F_\O(\beta)$ is even. By hypothesis, we know that $F(\delta)$ is a multiple of $n/\kappa$, and  $n/\kappa$ must be odd, so we can conclude that $F(\delta) \equiv 0$ (mod $2n/\kappa$).    On the other hand,  Lemma \ref{L:F3} implies that for all divisors $\beta \in \CP_3$ we have  $F(\beta)$ is a multiple of $n/\kappa$, so  we can similarly conclude that $F(\delta_3)$ is a multiple of $2n/\kappa$ as well.   It is an immediate consequence that $\hat{\delta} \equiv 0$ (mod $2n/\kappa$), showing that $\hat{\delta}$ satisfies all of the conditions of Theorem \ref{T:P12}.  In particular, this shows that $\hat{\delta} \in \CP_1+\CP_2$ and thus $\delta \in \CP$.
\vskip .1in

\noindent {\bf Case 2:} $n$ odd, all orbits are of Type III

This case works in a similar manner to what we did above.  In order to define our divisor $\delta_3$ we first note that If every orbit of $G$ is of Type III, then we have

\begin{eqnarray*}
\alpha &=& \sum_\O a_\O \\
&=& \sum_\O \frac{\sum \delta(x_i)- \sum\delta(y_i)}{n} \\
&=& \sum_\O \frac{\sum \delta(x_i)- \sum\delta(y_i)}{n} + \frac{\sum \delta(x_i)+ \sum\delta(y_i)}{n}\\
&=& \sum_\O \frac{2\sum \delta(x_i)}{n}
\end{eqnarray*}

\noindent and in particular $\alpha$ will be even.   We choose one orbit and designate it $\Omega$.  For the vertices in the $x_i$ suborbit of $\Omega$, we let $\delta_3$ have the value $k_\Omega(a_\Omega-\alpha/2)$ and for the vertices in the $y_i$-suborbit we give $\delta_3$ the value $-k_\Omega\alpha/2$.  On all orbits $\O \ne \Omega$, we define $\delta_3(x_i) =k_\O a_\O$ and $\delta_3(y_i) = 0$ for all $i$. It is easy to check that $\delta_3$ has the necessary symmetry and divisibility properties to be an element of $\CP_3$.   Moreover, one can go through very similar computations as in Case 1 to see that $\hat{\delta} = \delta-\delta_3$ satisfies the conditions of Theorem \ref{T:P12} and thus $\delta \in \CP$.
\end{proof}

\begin{theorem}\label{T:Pevena}
Assume $n$ and $\kappa$ are both even. Then a divisor $\delta \in \CP_1+\CP_2+\CP_3$ if and only if $\delta$ satisfies the following conditions:
\begin{itemize}
\item $\sum_{v} \delta(v)=0$
\item If $\O$ is an orbit of Type III  then $\sum \delta(x_i)\equiv \sum\delta(y_i)$ (mod $n$).
\item $\sum_{v \in \O} \delta(v)$ is even for all inertial orbits.
\item $F(\delta)$ is a multiple of $2n/\kappa$.
\end{itemize}
\end{theorem}

\begin{proof}

The proof of this Theorem is similar to that of Theorem \ref{T:Podd}.  In particular, it is straightforward to check that all four conditions hold for divisors that are in $\CP_1+\CP_2$, and the first three conditions are again straightforward for divisors in $\CP_3$.  To see the fourth condition must hold for divisors in $\CP_3$, we note that Lemma \ref{L:F3} implies that if $\delta_3 \in \CP_3$ then  $F_\O(\delta_3)$ is a multiple of $n$ for all orbits $\O$.  Moreover, because $\kappa$ is even we see that $n=\frac{2n}{\kappa}\cdot \frac{\kappa}{2}$ and therefore that $F_\O(\delta_3)$ is a multiple of $2n/\kappa$.

To see that the conditions are sufficient, we define $\delta_3(x_i) =k_\O a_\O$ at all vertices in the $x_i$-suborbit of each orbit $\O$ of Type III, where $a_\O$ is defined as above.  If there are any orbits of Type I, we choose one and designate it  $\Omega$, and as above we define $\delta_3(z_i)=-k_\Omega \alpha$ for each vertex in $\Omega$.  For all other vertices, we define $\delta_3(v)=0$.   The proof of Case 1 of Theorem \ref{T:Podd} then carries through with minor modifications.

If all orbits have Type III then we can follow the argument from Case 2 of Theorem \ref{T:Podd} by choosing one orbit $\Omega$ and setting $\delta_3(\omega)=k_\Omega(a_\Omega-\alpha/2)$ for the vertices in the $x_i$-suborbit and $\delta_3(\omega)=-k_\Omega\alpha/2$ for the vertices in the $y_i$-suborbit of $\Omega$.  As above, we define $\delta_3(x_i) =k_\O a_\O$ and $\delta_3(y_i) = 0$ for the orbits $\O \ne \Omega$.  It is clear that $\delta_3 \in \CP_3$, and it is straightforward to check that $\hat{\delta}= \delta-\delta_3$ satisfies the conditions of Theorem \ref{T:P12}.  The theorem follows.
 \end{proof}

\begin{theorem} \label{T:Pevenb}
Assume $n$ is  even and $\kappa$ is odd.  Then a divisor $\delta \in \CP_1+\CP_2+\CP_3$ if and only if $\delta$ satisfies the following conditions:
\begin{itemize}
\item $\sum_{v} \delta(v)=0$
\item If $\O$ is an orbit of Type III  then $\sum \delta(x_i)\equiv \sum\delta(y_i)$ (mod $n$).
\item If $\O$ is an inertial orbit, we have  $\sum_{v \in \O} \delta(v)$ is even.
\item $F(\delta)$ is a multiple of $n/\kappa$.
\end{itemize}
\end{theorem}
\begin{proof}

As above, we first check that the conditions are all necessary.  The proof is identical to the previous theorem except that if $\kappa$ is odd then we have that $n \equiv \frac{n}{\kappa}$ mod $\frac{2n}{\kappa}$, which means that we can only conclude that for divisors $\delta_3 \in \CP_3$ we have that $F(\delta_3)$ is a multiple of $n/\kappa$ rather than $2n/\kappa$.

In order to check the sufficiency of these conditions, we consider two cases.  If it is the case that $F(\delta)$ is a multiple of $2n/\kappa$ then we are able to use the same constructions as in the proof of Theorem \ref{T:Pevena} and the proofs will carry through exactly.

It remains to consider the case where  $F(\delta) \equiv n/\kappa$ (mod $2n/\kappa$).   We note that if this is the case then there must be at least one Type II orbit, which in turn implies the existence of at least one Type I orbit.  We wish to choose one Type I orbit and designate it by $\Omega$ and  one Type II orbit and designate it by $\Theta$.  As in the previous cases, we define $\delta_3(x_i) =k_\O a_\O$ at all vertices in the $x_i$-suborbit of each orbit $\O$ of Type III and $\delta_3(y_i)=0$ for the vertices in the other suborbits.  For each vertex $\theta \in \Theta$ we define $\delta_3(\theta) = k_\Theta$, and for each $\omega \in \Omega$  we set $\delta_3(\omega)=-k_\Omega (\alpha+1)$.  Finally, we set $\delta_3(v)=0$ for all vertices in Type I and II orbits other than $\Omega$.    It is clear that this defines a divisor in $\CP_3$, and moreover we can compute:

\begin{eqnarray*}
F(\delta_3) &=& \sum_{\O \text{ Type III}} a_\O n \left(\frac{n}{k_\O}+1 \right) + n \left(\frac{n}{k_\Theta}+2 \right) - (\alpha + 1)  n \left(\frac{n}{k_\Omega}+1 \right) \\
&=&\sum_{\O \text{ Type III}}n a_\O\left(\frac{n}{k_\O}-\frac{n}{k_\Omega} \right) + n \left(\frac{n}{k_\Theta}-\frac{n}{k_\Omega} \right) + n \\
& \equiv & n \text{ (mod $\frac{2n}{\kappa}$)} \\
& \equiv &  \frac{n}{\kappa}  \text{ (mod $\frac{2n}{\kappa}$)} \\
\end{eqnarray*}

In particular, if we define $\hat{\delta} = \delta - \delta_3$ then we see that $F(\hat{\delta})$ is a multiple of $2n/\kappa$.  It is straightforward to check that $\hat{\delta}$ satisfies the other conditions in the statement of Theorem \ref{T:P12}, which proves the Theorem.
\end{proof}

We now wish to use the structure of $\CP$ to analyze the group $\CD/\CP$.

\begin{example}
Before proving the general case, we return to the case of the graph $K_{4,4}$ discussed in Example \ref{E:K44}.  In particular, we have $n=\kappa=4$, and therefore according to Theorem \ref{T:Pevena}, a divisor of total degree zero will be in $\CP$ if and only if we have:
\begin{itemize}
\item $\delta(x)\equiv \delta(y)$ (mod $4$).
\item $\delta(w_1)+\delta(w_2)+\delta(w_3)+\delta(w_4)$ is even, and
\item $\delta(z_1) + \delta(z_2)$ is even.
\end{itemize}
In particular, the divisor must satisfy one condition mod $4$ and two conditions mod $2$.  However, we note that the first condition implies that $\delta(x)$ and $\delta(y)$ have the same parity and in particular that $\delta(x)+\delta(y)$ will be even. Because the total degree of the divisor is zero, the third condition will therefore follow immediately from the first two.  This implies that the group $\CD/\CP$ is isomorphic to $(\ZZ/4\ZZ) \oplus (\ZZ/2\ZZ)$.
\end{example}

More generally, a similar phenomenon will occur and some of the conditions will be automatically satisfied if the others are.  Explicitly, we have the following.

\begin{theorem}\label{T:DP}
Let $\CD$ be the set of divisors of degree zero on $G$.  Then
\[\CD/\CP \cong \begin{cases}\ZZ/\frac{n}{\kappa}\ZZ \oplus ( \ZZ/n\ZZ)^{t_3}\text{ if $n$ odd}\\
\ZZ/\frac{n}{\kappa}\ZZ \oplus( \ZZ/n\ZZ)^{t_3} \oplus(\ZZ/2\ZZ)^{\tilde{t}}\text{ if $n$ even and $\kappa$ is odd} \\
\ZZ/\frac{2n}{\kappa}\ZZ \oplus( \ZZ/n\ZZ)^{t_3} \oplus(\ZZ/2\ZZ)^{\tilde{t}}\text{ if $n$ and $\kappa$ are both even}
\end{cases}\]
\noindent where $\tilde{t} = \min(t_1-1,0)+\min(t_2-1,0)$.
\end{theorem}

\begin{proof}
Let $\delta$ be a divisor in $\CD$.  For each inertial orbit $\O$, we define the map $A_\O:\CD \rightarrow \ZZ/2\ZZ$ by setting $A_\O(\delta) = \sum_{v \in \O} \delta(v)$.  On the other hand, if $\O=\{x_i\} \cup \{y_i\}$ is a Type III orbit then we define $A_\O:\CD \rightarrow \ZZ/n\ZZ$ by setting $A_\O = \sum(x_i)-\sum(y_i)$.

If $n$ is odd then we define the surjective map $\alpha: \CD \rightarrow \ZZ/(n/\kappa)\ZZ \oplus ( \ZZ/n\ZZ)^{t_3}$ by setting \[\alpha(\delta) =\left (F(\delta), \bigoplus_{\O \text{ Type III}} A_\O(\delta)\right)\]
\noindent and we note that Theorem \ref{T:Podd} gives us that the kernel of $\alpha$ is exactly the elements of $\CP$, proving the theorem in this case.

If $n$ is even, we set $\epsilon=1$ (resp. $0$) if $\kappa$ is odd (resp. even). Similar to the above, we define the map $\alpha: \CD \rightarrow \ZZ/\frac{\epsilon n}{\kappa}\ZZ \oplus ( \ZZ/n\ZZ)^{t_3} \oplus (\ZZ/2\ZZ)^{t_1+t_2}$ by setting \[\alpha(\delta) = \left(F(\delta), \bigoplus_{\O \text{ Type III}} A_\O(\delta),\bigoplus_{\O \text{ Type I}} A_\O(\delta), \bigoplus_{\O \text{ Type II}} A_\O(\delta)\right )\]
\noindent and once again it is clear that this map is surjective.  The content of Theorems \ref{T:Pevena} and \ref{T:Pevenb} is that elements of $\CP$ will be in the kernel of $\alpha$, but in this case these conditions are not sufficient. In particular, we note that if $\delta \in \CP$ then $F(\delta)$ will be even, and therefore we have that
\begin{eqnarray*}
0 & \equiv &  F(\delta) \\
& \equiv & \sum_{\O \text{ Type II}} F_\O(\delta)\\
& \equiv & \sum_{\O \text{ Type II}}\left(\sum_{v \in \O} \delta(v)\right) \\
& \equiv & \sum_{\O \text{ Type II}} A_\O(\delta)\text{ (mod $2$)}\\
\end{eqnarray*}

Moreover, we know that for a divisor $\delta \in \CP$ and a Type III orbit $\O=\{x_i\} \cup \{y_i\}$ we have $\sum \delta(x_i) \equiv \sum \delta(y_i)$ mod $n$, and because $n$ is even this implies that $\sum_{v \in \O} \delta(v)$ is even.  Because we know that the total degree of $\delta$ is equal to zero, these facts further imply that $\sum_{\O \text{ Type I}} A_\O(\delta)$ is even as well.  This implies that $\CD/\CP \cong \ZZ/\frac{\epsilon n}{\kappa}\ZZ \oplus( \ZZ/n\ZZ)^{t_3} \oplus(\ZZ/2\ZZ)^{\tilde{t}}$, as desired.
\end{proof}

\section{Intersections of Pullbacks}

In this section, we wish to determine the structure of $\CK$, which we defined to be the kernel of the map $\Phi: \Jac(H_1) \oplus \Jac(H_2) \oplus \Jac(H_3) \rightarrow \CP/\CL'$.  In some sense, the previous section computed $\CP_1  + \CP_2 + \CP_3$ and in this section we wish to understand the difference between that set and  $\CP_1 \oplus \CP_2 \oplus \CP_3$.   In order to do this, we want to look at the intersections of the groups  $\CP_i$.

\begin{example}
Considering the graph in Example \ref{E:K44}, we display in Figure \ref{F:Kint} the generic elements of both $\CP_0$ and $\CP_1 \cap \CP_2$.

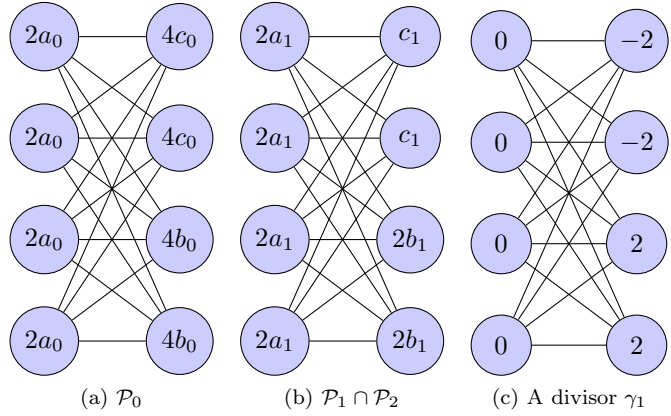
\begin{figure}[!htbp]
\centering
\subfloat[$\CP_0$]{\begin{tikzpicture}
  [scale=.9,auto=left,every node/.style={state,minimum size=.8cm,fill=blue!20}]
  \node(x1) at (-1,1) {$2a_0$};
  \node(x2) at (-1,2.5){$2a_0$};
  \node(x3) at (-1,4){$2a_0$};
  \node(x4) at (-1,5.5) {$2a_0$};
  \node(y1) at (1,1){$4b_0$};
  \node(y2) at (1,2.5){$4b_0$};
  \node(z1) at (1,4) {$4c_0$};
  \node(z2) at (1,5.5){$4c_0$};
  \foreach \from/\to in {x1/y1,x1/y2,x1/z1,x1/z2,x2/y1,x2/y2,x2/z1,x2/z2,x3/y1,x3/y2,x3/z1,x3/z2,x4/y1,x4/y2,x4/z1,x4/z2}
    \draw (\from) -- (\to);
\end{tikzpicture}} \quad
\subfloat[$\CP_1 \cap \CP_2$]{\begin{tikzpicture}
  [scale=.9,auto=left,every node/.style={state,minimum size=.8cm,fill=blue!20}]
  \node(x1) at (-1,1) {$2a_1$};
  \node(x2) at (-1,2.5){$2a_1$};
  \node(x3) at (-1,4){$2a_1$};
  \node(x4) at (-1,5.5) {$2a_1$};
  \node(y1) at (1,1){$2b_1$};
  \node(y2) at (1,2.5){$2b_1$};
  \node(z1) at (1,4) {$c_1$};
  \node(z2) at (1,5.5){$c_1$};
  \foreach \from/\to in {x1/y1,x1/y2,x1/z1,x1/z2,x2/y1,x2/y2,x2/z1,x2/z2,x3/y1,x3/y2,x3/z1,x3/z2,x4/y1,x4/y2,x4/z1,x4/z2}
    \draw (\from) -- (\to);
\end{tikzpicture}}
\quad
\subfloat[A divisor $\gamma_1$]{\begin{tikzpicture}
  [scale=.9,auto=left,every node/.style={state,minimum size=.8cm,fill=blue!20}]
  \node(x1) at (-1,1) {$0$};
  \node(x2) at (-1,2.5){$0$};
  \node(x3) at (-1,4){$0$};
  \node(x4) at (-1,5.5) {$0$};
  \node(y1) at (1,1){$2$};
  \node(y2) at (1,2.5){$2$};
  \node(z1) at (1,4) {$-2$};
  \node(z2) at (1,5.5){$-2$};
  \foreach \from/\to in {x1/y1,x1/y2,x1/z1,x1/z2,x2/y1,x2/y2,x2/z1,x2/z2,x3/y1,x3/y2,x3/z1,x3/z2,x4/y1,x4/y2,x4/z1,x4/z2}
    \draw (\from) -- (\to);
\end{tikzpicture}}
 \caption{Understanding the relationship between $\CP_0$ and $\CP_1 \cap \CP_2$ for the $D_4$ action on $K_{4,4}$}
\label{F:Kint}
\end{figure}

We note that a divisor in $\CP_1 \cap \CP_2$ must have total degree zero, so it follows that $8a_1+4b_1+2c_1=0$, from which we can conclude that $c_1=-4a_1-2b_1$ must also be even.  In particular, any element of $\CP_1 \cap \CP_2$ is either an element of $\CP_0$ or differs from one by the divisor $\gamma_1$ depicted in Figure \ref{F:Kint}.  Thus, we have that  $\CP_1 \cap \CP_2 \cong \CP_0 \oplus \ZZ/2\ZZ$.

\end{example}

\begin{definition}
 Let $\O$ be an orbit of Type III. We define $\delta_\O$ to be the divisor which has value $1$ at all vertices in $\O$ and value $0$ at all vertices outside of $\O$.  Similarly, if $\O$ is an inertial orbit then we define $\delta_\O$ to be the divisor which has value $2$ at all vertices in $\O$ and value $0$ at all vertices outside of $\O$.  Set $\CQ$ to be the set of divisors $\{\sum s_\O \delta_\O |s_\O \in \ZZ\}$.   We note that divisors in $\CQ$ all possess the necessary symmetry and parity conditions to be in both $\CP_1$ and $\CP_2$,  although in general they do not also have total degree equal to zero.
\end{definition}

Recall that  $\kappa$ is the least common multiple of the indices $k_\O$, and in particular we can see that $2n/\kappa$ is the greatest common divisor of the sizes of all orbits of Type III and double the sizes of all orbits of Type I and II.  Thus, there is some divisor $\gamma = \sum s_\O \delta_\O \in \CQ$ so that the total degree of $\gamma$ is equal to $2n/\kappa$, and we wish to fix one such divisor.  One can compute that this is equivalent to the numerical condition that $\sum_\O \frac{s_\O}{k_\O} = \frac{1}{\kappa}$. We note that, if there is an orbit $\Omega$ whose index is equal to $\kappa$ then we may choose $\gamma=\delta_\Omega$.

\begin{lemma}
We have that  $(\CP_1 \cap \CP_2) \oplus (\ZZ/\kappa\ZZ) \cong \CP_0 \oplus \left(\bigoplus_{\O} \ZZ/k_\O\ZZ\right)$.
\end{lemma}

\begin{proof}
It follows from the definitions and the fact that $\sigma_1$ and $\sigma_2$ generate all of $D_n$ that $\CP_1 \cap \CP_2$ consists of divisors on $G$ of total degree zero that are constant on all orbits and take on even values on inertial orbits.  In particular, $\CP_1 \cap \CP_2 = \CQ \cap \CD$.  On the other hand, it follows from the definition  that $\CP_0 = \langle k_\O\delta_\O \rangle \cap \CD$.

For any orbit $\O$, define the divisor $\gamma_\O=\delta_\O- \frac{\kappa}{k_\O}\gamma$.  One can check that this divisor has total degree zero and therefore is an element of $\CQ \cap \CD$.  Moreover, it is clear that $k_\O\gamma_\O \in \CP_0$.  We define a map $\Phi :  \CP_0 \oplus \left( \bigoplus_{\O} \ZZ/k_\O\ZZ\right) \rightarrow (\CP_1 \cap \CP_2)$: \[\Phi(\delta, \oplus i_\O) = \delta + \sum_\O i_\O\gamma_\O \in \CP_1 \cap \CP_2.\]  This map is surjective, but it is not quite injective. In particular, we compute:
 \begin{eqnarray*}
 \Phi(0, \oplus s_\O) &=& \sum_\O s_\O\gamma_\O\\
 &=& \sum_\O s_\O \delta_\O - \sum_\O s_\O \frac{\kappa}{k_\O} \gamma \\
 &=& \gamma - \gamma \kappa \sum_\O \frac{s_\O}{k_\O}\\
 &=& \gamma - \gamma\\
 &=&0
 \end{eqnarray*}
One checks that this is the unique choice of elements $\oplus i_\O$ in the kernel of $\Phi$.  By the Chinese Remainder Theorem, $\bigoplus_{\O} \ZZ/k_\O\ZZ \cong \ZZ/\kappa\ZZ$, proving the result.
\end{proof}

Next, we will consider the set $(\CP_1 + \CP_2) \cap \CP_3$.

\begin{example}
It follows from our calcuations above that in the case of $K_{4,4}$ defined in Example \ref{E:K44}, a divisor of total degree zero will be in $(\CP_1 + \CP_2) \cap \CP_3$ if and only if it is constant on all orbits and each entry is a multiple of the index of the orbit.  On the other hand, a divisor will be in $\CP_0$ if it is constant on all orbits, each $\delta(w_i)$ is even, and each $\delta(z_i)$, $\delta(x)$, and $\delta(y)$ is a multiple of four.  Because the total degree is zero, this allows us to conclude that $(\CP_1 + \CP_2) \cap \CP_3 \cong \CP_0 \oplus (\ZZ/2\ZZ)$.
\end{example}

More generally, we can prove the following formula:

\begin{lemma}
Assume that $G$ has $t_1$ orbits of Type I and $t_2$ orbits of Type II.  Recall that we set $\tilde{t} = \min(t_1-1,0)+\min(t_2-1,0)$ Then we have the following:
\[((\CP_1 + \CP_2) \cap \CP_3) \cong \begin{cases} \CP_0 \text{ if $n$ odd} \\ \CP_0 \oplus (\ZZ/2\ZZ)^{\tilde{t}} \text{ if $n$ even and $\kappa$ is odd} \\ \CP_0 \oplus (\ZZ/2\ZZ)^{\tilde{t}+1} \text{ if $n$ and $\kappa$ are both even} \end{cases}\]
\end{lemma}

\begin{proof}
Consider a divisor $\delta$ in $\CP_0$.  It follows from the remarks in Definition \ref{D:sets}  that the value of the divisor is constant on all vertices in an orbit; we denote this value by $\delta(\O)$.  Moreover, we have that $(2n/|\O|) | \delta(\O)$, and thus $\sum_{v \in \O} \delta(v)$ is a multiple of $2n$ and in particular must be even.  Moreover, the fact that $\delta$ is constant on all orbits means that $\sum \delta(x_i)=\sum\delta(y_i)$ for all Type III orbits.  Similar to Lemma \ref{L:F3}, we can compute $F_\O(\delta)$ for each orbit $\O$ as follows:

\begin{alignat*}{4}
\O \text{ Type I: } & F_\O(\delta)  &=& \left(\frac{n}{k_\O}\right) \left(\frac{n}{k_\O}+1\right)\delta(\O) \\
\O \text{ Type II: } &  F_\O(\delta)  &=& \left(\frac{n}{k_\O}\right) \left(\frac{n}{k_\O}+2\right)\delta(\O) \\
\O \text{ Type III: } &  F_\O(\delta) &=& 2 \left(\frac{n}{k_\O}\right) \left(\frac{n}{k_\O}+1\right)\delta(\O) \\
\end{alignat*}

In all three cases, we see that $F_\O(\delta)$ is a multiple of $2n/k_\O$ and therefore a multiple of $2n/\kappa$, so it must be the case that $F(\delta)$ is a multiple of $2n/\kappa$.  Thus, $\delta$ satisfies all three conditions in Theorem \ref{T:P12} and must be an element of  $\CP_1 + \CP_2$.   On the other hand, it follows immediately from the remarks in Definition \ref{D:sets} that any divisor in $\CP_0$ is also in $\CP_3$.  Therefore, we can see that $\CP_0 \subseteq (\CP_1 + \CP_2) \cap \CP_3$.

Conversely, consider a divisor $\delta \in (\CP_1 + \CP_2) \cap \CP_3$.  Because  $\delta \in \CP_3$, it must be constant on all inertial orbits as well as each of the suborbits of a non-inertial orbit.  We note that it follows from Condition (2) of Theorem \ref{T:P12} that in order to be in $\CP_1+\CP_2$ we must have that $\sum \delta(x_i)=\sum\delta(y_i)$, from which we can conclude that $\delta$ is constant on orbits of Type III as well.  As above, we will use $\delta(\O)$ to denote the value of the divisor at all vertices in $\O$. We further know from the fact that $\delta \in \CP_3$ that for each orbit we have $k_\O | \delta(\O)$.

Let $n$ be odd and let $\O$ be an orbit of Type I or Type II. Because $\delta \in \CP_1+\CP_2$ we must have that $\sum_{v \in \O} \delta(v)$ is even.  However, $\sum_{v \in \O} \delta(v)=n\delta(\O)=nk_\O a_\O$ for some integer $a_\O$.  Because $n$ is odd, $k_\O$ must be as well, which implies that $a_\O$ is even.  In particular, this tells us that $(2n/|\O|)$ divides $\delta(\O)$ for all orbits.  Combining this with the results of the previous paragraph, we see that $\delta \in \CP_0$ and in particular this proves that if $n$ is odd that $\CP_0 = (\CP_1 + \CP_2) \cap \CP_3$.

On the other hand, if $n$ is even then $(\CP_1 + \CP_2) \cap \CP_3$ contains divisors that are not in $\CP_0$.  In particular, let us assume that there are $t_1 \ge 1$ orbits of Type I, and designate one of these orbits as $\Omega$.  For each Type I orbit $\O \ne \Omega$ we define the divisor $\gamma_\O$ so that $\gamma_\O(v) = k_\O$ for all $v \in \O$, $\gamma_\O(\omega) = -k_\Omega$ for each $\omega \in \Omega$, and $\gamma_\O(v)=0$ for all other vertices. One quickly checks that the total degree of this divisor is $0$, that it is constant on all orbits, and that the sum on all orbits is even.  Moreover, we check that $F_\O(\gamma_\O)+F_{\Omega}(\gamma_\O) = n(\frac{n}{k_\O}-\frac{n}{k_\Omega})$ which will be a multiple of $n/\kappa$.  This shows that $\gamma_\O \in (\CP_1 + \CP_2) \cap \CP_3$.  One can check that $\gamma_\O \nin \CP_0$ but $2\gamma_\O \in \CP_0$.

Similarly, if there is more than one orbit of Type II then we can designate one of them to be $\Theta$ and for all other orbits of Type II we define $\gamma_\O$ so that $\gamma_\O(v) = k_\O$ for all $v \in \O$, $\gamma_\O(\theta) = -k_\Theta$ for each $\theta \in \Theta$, and $\gamma_\O(v)=0$ for all other vertices.  Once again, we will see that $\gamma_\O \in (\CP_1 + \CP_2) \cap \CP_3$ and $\gamma_\O \nin \CP_0$ but $2\gamma_\O \in \CP_0$. Finally, we define $\gamma_\infty$ to be the divisor which has value $\gamma_\O(\theta) = -k_\Theta$ for each $\theta \in \Theta$ and $\gamma_\O(\omega) = k_\Omega$ for each $\omega \in \Omega$.  As above, one sees that $\gamma_\infty \nin \CP_0$ but $2\gamma_\infty \in \CP_0$.  It is also clear that $\gamma_\infty \in \CP_3$ and satisfies the first two conditions required by Theorem \ref{T:P12} in order to be in $\CP_1+\CP_2$.  To check the third condition, we compute that
\begin{eqnarray*}
F(\gamma_\infty) &=& F_\Theta(\gamma_\infty)+F_\Omega(\gamma_\infty)\\
&=&n \left(\frac{n}{k_\Theta}+2\right) - n \left(\frac{n}{k_\Omega}+1 \right)\\
&=& n \left(\frac{n}{k_\Theta} - \frac{n}{k_\Omega} + 1\right)
\end{eqnarray*}

If $\kappa$ is even then $n$ is a multiple of $2n/\kappa$ and therefore $F(\gamma_\infty)$ is as well, implying that $\gamma_\infty \in \CP_1+\CP_2$ as well.  On the other hand, if $\kappa$ is odd then $k_\Theta$ and $k_\Omega$ must both be odd, which in turn implies that $\left(\frac{n}{k_\Theta} - \frac{n}{k_\Omega} + 1\right)$ is odd.  Moreover, we have already seen that $n \equiv \frac{n}{\kappa}$ (mod $2n/\kappa$), and in particular we will see that $F(\gamma_\infty)$ is not a multiple of $2n/\kappa$, so in this case $\gamma_\infty$ is not an element of $\CP_1+\CP_2$.

Every divisor in $(\CP_1 + \CP_2) \cap \CP_3$ can be written as the sum of a divisor in $\CP_0$ and some combination of the $\gamma_\O$ (including, if $\kappa$ is even, $\gamma_\infty$).  The lemma follows.
\end{proof}

 We are now ready to recall from \cite{GM} the definition of the natural surjection $\Phi: \Jac(H_1) \oplus \Jac(H_2) \oplus \Jac(H_3) \rightarrow \CP/\CL'$  discussed earlier.  We set $\phi_i:G \rightarrow H_i$ to be the natural surjection, and $\phi_i^*$ to be the pullback map associated to this map.  By definition, we have that $\Jac(H_i) = \CD_i/\CL_i$, so we can define maps $\Phi_i:\Jac(H_i) \rightarrow \CP/\CL'$ by letting $\Phi_i(d_i+\CL_i) = \phi_i^*(d_i)+\CL'$. One can easily check that this map is well defined, as the pullback of a divisor in $\CL_i$ will be in $\CL'$.  We then define $\Phi$ as follows:

\begin{definition}\label{D:Phi}
Let $\Phi: \Jac(H_1) \oplus \Jac(H_2) \oplus \Jac(H_3) \rightarrow \CP/\CL'$ be the map defined by $\Phi(\delta_1,\delta_2,\delta_3)=\Phi_1(\delta_1)+\Phi_2(\delta_2)+\Phi_3(\delta_3)$.
\end{definition}

Any element of $\CP$ can be decomposed as the sum of pullbacks of elements on the three quotient graphs, and therefore this map is surjective.  Determining its kernel is the contents of the following theorem.

\begin{theorem}\label{T:K}
Assume that $G$ is a graph with a harmonic $D_n$-action so that $G/D_n=\hat{G}$. We define the map $\Phi: \Jac(H_1) \oplus \Jac(H_2) \oplus \Jac(H_3) \rightarrow \CP/\CL'$ as above, setting  the kernel of $\Phi$ equal to the set $\CK$.  Then
\[\CK \oplus (\ZZ/\kappa\ZZ) \cong \begin{cases} \Jac(\hat{G})^2\oplus\left(\bigoplus_{\O}( \ZZ/k_\O\ZZ)\right) \text{ if $n$ odd} \\ \Jac(\hat{G})^2\oplus (\ZZ/2\ZZ)^{\tilde{t}}\oplus\left(\bigoplus_{\O}( \ZZ/k_\O\ZZ)\right) \text{ if $n$ even and $\kappa$ is odd} \\ \Jac(\hat{G})^2\oplus (\ZZ/2\ZZ)^{\tilde{t}+1}\oplus\left(\bigoplus_{\O}( \ZZ/k_\O\ZZ)\right) \text{ if $n$ and $\kappa$ are both even} \end{cases}\]
\end{theorem}

\begin{proof}
We note that the only elements in the kernel of the map come from situations where elements in $\CP$ can be decomposed as elements of the $\CP_i$ in more than one way.  In particular, $\Phi$ will be injective if and only if the sets $\CP_1, \CP_2, \CP_3$ are all strongly disjoint in the sense that for any permutation of the three sets we have that $\CP_i \cap (\CP_j+\CP_k) = \{0\}$.

More explicitly, we know that the kernel of this map can be computed as having a piece coming from $\CP_1 \cap \CP_2$ and another piece coming from $(\CP_1 + \CP_2) \cap \CP_3$.  Computing these sets was the content of the previous two lemmata, and the Theorem is an immediate consequence.
\end{proof}

\begin{example}
In the case of our running example of the bipartite graph $K_{4,4}$, we note that it is easy to compute the various quotient graphs, as depicted in Figure \ref{F:Kquot}.  In particular,  $K_{4,4}/D_4$ is a tree and  it is straightforward to compute the critical groups $\Jac(K_{4,4}/\langle \sigma_1 \rangle) \cong \Jac(K_{4,4}/\langle \sigma_2 \rangle) \cong \ZZ/4\ZZ \oplus \ZZ/8\ZZ$ and $\Jac(K_{4,4}/\langle \tau \rangle) \cong \ZZ/2\ZZ$.
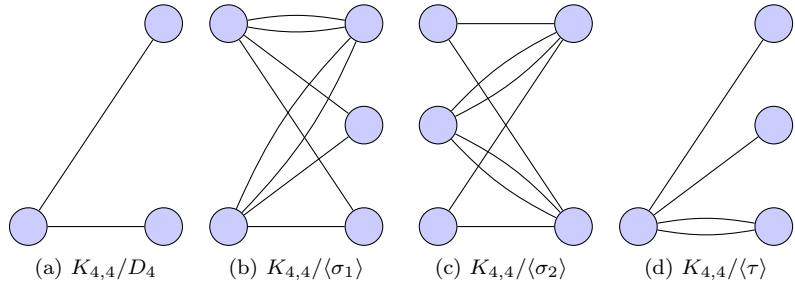
\begin{figure}[!htbp]
\centering
\subfloat[$K_{4,4}/D_4$]{\begin{tikzpicture}
  [scale=.9,auto=left,every node/.style={state,minimum size=.5cm,fill=blue!20}]
  \node(z) at (-1,1) {};
  \node(w) at (1,1){};
  \node(xy) at (1,4) {};
  \foreach \from/\to in {w/z, z/xy}
    \draw (\from) -- (\to);
\end{tikzpicture}} \quad
\subfloat[$K_{4,4}/\langle \sigma_1 \rangle$]{\begin{tikzpicture}
  [scale=.9,auto=left,every node/.style={state,minimum size=.5cm,fill=blue!20}]
  \node(z14) at (-1,1) {};
  \node(z23) at (-1,4){};
   \node(w1) at (1,1){};
  \node(w2) at (1,2.5){};
  \node(xy) at (1,4) {};

  \foreach \from/\to in {z14/w1,z14/w2,z23/w1,z23/w2}
    \draw (\from) -- (\to);
   \foreach \from/\to in {z23/xy,xy/z23,z14/xy,xy/z14}
    \draw (\from) to[bend right=10] (\to);

\end{tikzpicture}}
 \quad
\subfloat[$K_{4,4}/\langle \sigma_2 \rangle$]{
\begin{tikzpicture}
  [scale=.9,auto=left,every node/.style={state,minimum size=.5cm,fill=blue!20}]
  \node(z1) at (-1,1) {};
  \node(z24) at (-1,2.5){};
  \node(z3) at (-1,4){};
  \node(w12) at (1,1){};
  \node(xy) at (1,4) {};
  \foreach \from/\to in {z1/w12,z3/w12,z1/xy,z3/xy}
    \draw (\from) -- (\to);
  \foreach \from/\to in {z24/w12,w12/z24,z24/xy,xy/z24}
    \draw (\from) to[bend right=10] (\to);
\end{tikzpicture}} \quad
\subfloat[$K_{4,4}/\langle \tau \rangle$]{\begin{tikzpicture}
  [scale=.9,auto=left,every node/.style={state,minimum size=.5cm,fill=blue!20}]
  \node(z) at (-1,1) {};
  \node(w12) at (1,1){};
  \node(x) at (1,2.5){};
  \node(y) at (1,4) {};
  \foreach \from/\to in {z/x,z/y}
    \draw (\from) -- (\to);
  \foreach \from/\to in {z/w12,w12/z}
    \draw (\from) to[bend right=10] (\to);
\end{tikzpicture}}

\caption{Quotient Graphs from the $D_4$ action on $K_{4,4}$}
\label{F:Kquot}
\end{figure}

We further recall that we have one orbit of each type (so in particular $\tilde{t}=0$) and that the indices of the three orbits are $1,2.$ and $4$.   Thus, this theorem tells us that $\CK \oplus \ZZ/4\ZZ \cong (\ZZ/2\ZZ) \oplus (\ZZ/2\ZZ) \oplus (\ZZ/4\ZZ)$.  In particular, this implies that there is a short exact sequence given by:
\[ 1 \rightarrow (\ZZ/2\ZZ)^2 \rightarrow \Jac(K_{4,4}/\langle \sigma_1 \rangle) \oplus \Jac(K_{4,4}/\langle \sigma_2 \rangle) \oplus \Jac(K_{4,4}/\langle \tau \rangle) \rightarrow \CP/\CL' \rightarrow 1\]
This allows us to conclude that $\CP/\CL'$ will be a subgroup of $(\ZZ/2\ZZ) \oplus (\ZZ/4\ZZ)^2 \oplus (\ZZ/8\ZZ)^2$ whose quotient is $(\ZZ/2\ZZ)^2$.
\end{example}

\section{Firing Vertices}\label{S:LL'}

Recall that the set $\CL$ consists of divisors on $G$ that correspond to firing at combinations of vertices in $G$, while the set $\CL'$ is generated by pulling back divisors on $G/\langle\sigma_1\rangle,G/\langle\sigma_2\rangle,$ and $G/\langle\tau\rangle$ that correspond to firing vertices on those graphs.  The goal of this section is to understand the relationship between $\CL$ and $\CL'$.

\begin{example}
If we represent divisors on  $K_{4,4}$ by column vectors corresponding to the transpose of $[\delta(w_1), \delta(w_2), \delta(w_3), \delta(w_4), \delta(z_1),\delta(z_2,\delta(x), \delta(y)]$ then one can easily determine that the set $\CL$ consists of all integer linear combinations of the divisors:
\[ \left\{ \left[
             \begin{array}{c}
               -4 \\
               0 \\
               0 \\
               0 \\
               1 \\
               1 \\
               1 \\
               1 \\
             \end{array}
           \right],\left[
             \begin{array}{c}
              0 \\
               -4 \\
               0 \\
               0 \\
               1 \\
               1 \\
               1 \\
               1 \\
             \end{array}
           \right],\left[
             \begin{array}{c}
               0 \\
               0 \\
               -4 \\
               0 \\
               1 \\
               1 \\
               1 \\
               1 \\
             \end{array}
           \right],\left[
             \begin{array}{c}
               0 \\
               0 \\
               0 \\
               -4 \\
               1 \\
               1 \\
               1 \\
               1 \\
             \end{array}
           \right], \left[
             \begin{array}{c}
               1 \\
               1 \\
               1 \\
               1 \\
               -4 \\
               0\\
               0 \\
               0 \\
             \end{array}
           \right],\left[
             \begin{array}{c}
               1 \\
               1 \\
               1 \\
               1 \\
               0 \\
               -4\\
               0 \\
               0 \\
             \end{array}
           \right]\left[
             \begin{array}{c}
               1 \\
               1 \\
               1 \\
               1 \\
               0 \\
               0\\
               -4 \\
               0 \\
             \end{array}
           \right]\left[
             \begin{array}{c}
               1 \\
               1 \\
               1 \\
               1 \\
               0 \\
               0\\
               0 \\
               -4 \\
             \end{array}
           \right]\right\}\]

On the other hand, the set $\CL'$ will be defined by looking at the pullbacks of the divisors corresponding to `firing' vertices on the quotient graphs as depicted in Figure \ref{F:Kquot}.  Explicitly, the pullbacks of the divisors corresponding to `firing' the vertices on $K_{4,4}/\langle \sigma_1 \rangle$ are given by :

\[  \left\{ \left[
             \begin{array}{c}
               -4 \\
               0 \\
               0 \\
               -4 \\
               2 \\
               2 \\
               2 \\
               2 \\
             \end{array}
           \right],\left[
             \begin{array}{c}
              0 \\
               -4 \\
               -4 \\
               0 \\
               2 \\
               2 \\
               2 \\
               2 \\
             \end{array}
           \right],\left[
             \begin{array}{c}
               2 \\
               2 \\
               2 \\
               2 \\
               0 \\
               0 \\
               -4 \\
               -4\\
             \end{array}
           \right],\left[
             \begin{array}{c}
               1 \\
               1 \\
               1 \\
               1 \\
               -4 \\
               0 \\
               0 \\
               0 \\
             \end{array}
           \right], \left[
             \begin{array}{c}
               1 \\
               1 \\
               1 \\
               1 \\
               0 \\
               -4\\
               0 \\
               0 \\
             \end{array}
           \right]\right\}\]
The set $\CL'$ will be generated by integer combinations of those vectors along with the set of vectors that are pullbacks of `firing' on the other two quotients:
\[ \left\{ \left[
             \begin{array}{c}
               -4 \\
               0 \\
               0 \\
               0 \\
               1 \\
               1 \\
               1 \\
               1 \\
             \end{array}
           \right],\left[
             \begin{array}{c}
              0 \\
               -4 \\
               0 \\
               -4 \\
               2 \\
               2 \\
               2 \\
               2 \\
             \end{array}
           \right],\left[
             \begin{array}{c}
               0 \\
               0 \\
               0 \\
               -4 \\
               1 \\
               1 \\
               1 \\
               1 \\
             \end{array}
           \right], \left[
             \begin{array}{c}
               2 \\
               2 \\
               2 \\
               2 \\
               -4 \\
               -4\\
               0 \\
               0 \\
             \end{array}
           \right],\left[
             \begin{array}{c}
               2 \\
               2 \\
               2 \\
               2 \\
               0 \\
               0\\
               -4 \\
               -4\\
             \end{array}
           \right],\left[
             \begin{array}{c}
               1 \\
               1 \\
               1 \\
               1 \\
               0 \\
               0\\
               -4 \\
               0 \\
             \end{array}
           \right],\left[
             \begin{array}{c}
               1 \\
               1 \\
               1 \\
               1 \\
               0 \\
               0\\
               0 \\
               -4 \\
             \end{array}
           \right], \left[
             \begin{array}{c}
               -4 \\
               -4 \\
               -4 \\
               -4 \\
               4 \\
               4\\
               4 \\
               4 \\
             \end{array}
           \right]\right\}\]
It is an elementary, if tedious, exercise to see that in this case $\CL = \CL'$.
\end{example}

\begin{lemma}\label{L:TypeIL}
If $v$ is a vertex in a Type I or Type II orbit then $\ell_v \in \CL'$.
\end{lemma}

\begin{proof}
Let $v$ be a vertex that is fixed by $\sigma_1$, and let $\phi_1: G \rightarrow H_1=G /\langle \sigma_1 \rangle$ be the natural quotient map.  With a possible abuse of notation, let us define $f_{x,y}$ to be the number of edges of $H_1$ between two vertices $x,y \in V(H_1)$ and let $\ell_{\phi_1(v)}$  be the divisor that corresponds to firing the vertex $\phi_1(v)$ on $H_1$.  We can compute:
\begin{eqnarray*}
\phi_1^*(\ell_{\phi_1(v)})&=&\sum_{y \in V(G_1)} \sum_{\substack{x \in V(G) \\ \phi(x)=y}} m_{\phi_1}(x) (\ell_{\phi_1(v)}(y) ) (x)\\
&=& \sum_{\substack{x \in V(G) \\ \sigma_1(x)=x}} 2 \ell_{\phi_1(v)} x + \sum_{\substack{\{x,y\} \subset V(G) \\ \sigma_1(x)=y}} \ell_{\phi_1(v)}(\phi_1(x))(x+y) \\
&=& -2 deg_{H_1}(\phi_1(v) )+ \sum_{\substack{x \in V(G)  \\ \sigma_1(x)=x \\ x \ne v}} 2 f_{\phi_1(v),\phi_1(x)} x + \sum_{\substack{\{x,y\} \subset V(G) \\ \sigma_1(x)=y}} f_{\phi_1(v),\phi_1(x)}(x+y) \\
&=&-deg_G(v) v + \sum_{\substack{x \in V(G) \\ x \ne v}} e_{v,x} x \\
&=&\ell_v
\end{eqnarray*}

This shows that $\ell_v \in \CL'$ for any vertex fixed by $\sigma_1$ and a similar argument will show that $\ell_v \in \CL'$ for any vertex fixed by $\sigma_2$.  However, we know that the set $\CL'$ is symmetric under the $D_n$-action, and therefore $\ell_v \in \CL'$ for all vertices in inertial orbits.
\end{proof}

\begin{theorem}\label{T:LL'}
$\CL/\CL'  \cong \bigoplus_\O \ZZ/(\frac{n}{k_\O}\ZZ)$, where the sum ranges over all non-inertial orbits.
\end{theorem}

\begin{proof}
The set $\CL$ is generated by the elements $\ell_v$ as $v$ ranges over all vertices $v \in V(G)$.  It follows from Lemma \ref{L:TypeIL} that for vertices in inertial orbits then $\ell_v \in \CL'$ as well, so it suffices to consider vertices in orbits of Type III.

Let $\O=\{x_i\} \cup \{y_i\}$ is an orbit of Type III and index $k_\O$.  One can compute in a manner similar to the proof of Lemma \ref{L:TypeIL} that  $\phi_1^*(\ell_{\phi_1(x_i)}) = \ell_{x_i} + \ell_{y_{n+1-i}}$, and therefore we have that  $\ell_{y_{n+1-i}} = -\ell_{x_i}  \in \CL/\CL'$ .  Moreover, by the symmetry of the $D_n$-action, one can see that for any pair $x_i, y_j$ we will have that $\ell_{y_{j}} = -\ell_{x_i}  \in \CL/\CL'$, and in particular the full contribution of the orbit $\O$ to $\CL'$  can be represented by the element $\ell_{x_1}$.   We further note that if we consider the natural quotient map $\phi_3: G \rightarrow H_3=G/ \langle \tau \rangle$ then we can compute that $\phi_3^*(\ell_{\phi_3(x_1)}) =  \sum_{i=1}^{n/k} \ell_{x_i} \in \CL' $.  In other words, the pullback of firing $\phi_3(x_1)$ corresponds to firing all of the vertices in the $x_i$-suborbit of $\O$.  In particular, this implies that $\frac{n}{k_\O} \ell_{x_1} \in \CL/\CL'$, which concludes the proof of the theorem.
\end{proof}

In particular, we note that $\CL = \CL'$ if there are no Type III orbits or if all Type III orbits are of index $k_\O = n$.  This is the case in the example of the $D_4$ action on $K_{4,4}$ that we have previously discussed.

\section{Summary of Results}\label{S:summary}

Now that we have been able to precisely determine the structure of many of the sets defined in Section \ref{S:Notation}, we wish to consider what this tells us about the Jacobian of the original graph $G$. Before we do this in general, let us revisit our running example one final time.

\begin{example}
Recall that in the case of the $D_4$ action on $K_{4,4}$ defined in Example \ref{E:K44}, we have shown the following facts:
\begin{itemize}
\item $\CD/\CP \cong \ZZ/4\ZZ \oplus \ZZ/2\ZZ$
\item $\CP/\CL'$ is a finite abelian group of order $2^{9}$
\item $\CL=\CL'$
\end{itemize}
Looking at the exact sequences in Theorem \ref{T:Exact}, we see that the fact that $\CL=\CL'$ implies that $\CD/\CP \cong \CD/ (\CP + \CL)$ and that we have a short exact sequence
\[ 1 \rightarrow \CP/\CL' \rightarrow \Jac(K_{4,4}) \rightarrow \CD/\CP \rightarrow 1\]
We can therefore conclude that $\Jac(K_{4,4})$ is a finite abelian group of order $2^{12}$. In fact, a more careful examination of these results would allow us to eliminate many of the groups of this order but these results do not allow us to specify the group exactly.  We note that it follows from a direct computation or from \cite{Lor1} that $\Jac(K_{4,4}) \cong (\ZZ/4\ZZ)^4 \oplus (\ZZ/16\ZZ)$.

\end{example}

More generally, we would like to revisit the exact sequences of Theorem \ref{T:Exact} in light of the results of the previous sections. Recall that $\tilde{t}=\min(t_1-1,0)+\min(t_2-1,0)$ appeared in the earlier results when $n$ is even; for notational simplicity, we set $\tilde{t}=0$ in the case where $n$ is odd.  Further recall that we defined $\epsilon$ to be $2$ if $n$ and $\kappa$ are both even, and $1$ otherwise. Then the results of the previous sections imply that we have the following exact sequences:

\begin{small}
\[ 1 \rightarrow (\CP \cap \CL)/\CL' \rightarrow \bigoplus_{\O \text{ Type III}} \ZZ/(\frac{n}{k_\O}\ZZ) \rightarrow (\ZZ/2\ZZ)^{\tilde{t}}\oplus( \ZZ/n\ZZ)^{t_3} \oplus \ZZ/(\frac{\epsilon n}{\kappa}\ZZ) \rightarrow \CD/(\CP + \CL) \rightarrow 1\]
\[ 1 \rightarrow (\CP \cap \CL)/\CL' \rightarrow \CP/\CL' \rightarrow \Jac(G) \rightarrow \CD/(\CP + \CL) \rightarrow 1\]
\[ 1 \rightarrow \ZZ/\kappa\ZZ \rightarrow \Jac(\hat{G})^2\oplus (\ZZ/2\ZZ)^{\tilde{t}+\epsilon-1}\oplus \left(\bigoplus_{\O}( \ZZ/k_\O\ZZ)\right) \rightarrow \bigoplus_{i=1,2,3} \Jac(H_i) \rightarrow \CP/\CL' \rightarrow 1\]
\end{small}

In order to help understand what this says about the relationship between $\Jac(G),  \Jac(\hat{G})$ and $\Jac(H_1) \oplus \Jac(H_2) \oplus \Jac(H_3)$ we first make the following definition:

\begin{definition}
Let $A$ (resp. $B$) be a finite abelian group, and for each prime $p$ let $A_p$ (resp. $B_p$) be its $p$-Sylow subgroup.  Then for any positive integer $m$ we say that $A$ and $B$ are $m$-equivalent if $|A_p|=|B_p|$ for all $p$ and moreover $A_p \cong B_p$ for all $p \nmid m$.
\end{definition}

We are now able to prove our main theorem, which we state in the following form:

\begin{theorem}\label{T:Equiv}
Let $G$ be a graph which admits a harmonic $D_n$-action.  Then  the group $\Jac(H_1) \oplus \Jac(H_2) \oplus \Jac(H_3)\oplus \ZZ/n\ZZ$ is $2n$-equivalent to the group  $\displaystyle \Jac(G) \oplus \Jac(\hat{G})^2\oplus \left(\bigoplus_{\O \text{ inertial}}( \ZZ/k_\O\ZZ)\right)$.
\end{theorem}

\begin{proof}
We note that the definition above implies that given an exact sequence
\[1 \rightarrow A_1 \rightarrow A_2 \rightarrow \ldots \rightarrow A_k \rightarrow 1\]
\noindent then we have that $\bigoplus_{i \text{ odd}}A_i$ is $m$-equivalent to $\bigoplus_{i \text{ even}}A_i$ as long as every prime dividing $|A_i|$ is also a divisor of $m$.  Therefore, if we use $\sim$ to denote $2n$-equivalence then the above exact sequences imply the following equivalences:

\[(\CP \cap \CL)/\CL' \oplus (\ZZ/2\ZZ)^{\tilde{t}}\oplus( \ZZ/n\ZZ)^{t_3} \oplus \ZZ/(\frac{\epsilon n}{\kappa}\ZZ) \sim \left( \bigoplus_{\O \text{ Type III}} \ZZ/(\frac{n}{k_\O}\ZZ)\right) \oplus \CD/(\CP + \CL)\]
\[(\CP \cap \CL)/\CL' \oplus \Jac(G) \sim \CP/\CL' \oplus \CD/(\CP + \CL)\]
\[\ZZ/\kappa\ZZ\oplus \left( \bigoplus_{i=1,2,3} \Jac(H_i)\right) \sim \Jac(\hat{G})^2\oplus (\ZZ/2\ZZ)^{\tilde{t}+\epsilon-1}\oplus \left(\bigoplus_{\O}( \ZZ/k_\O\ZZ)\right) \oplus \CP/\CL'\]

If we start with the third equation and add copies of $\ZZ/\frac{n}{\kappa}\ZZ$ and $\CD/(\CP + \CL)$ to both sides and then plug in the second equation to the third, we obtain that $\CD/(\CP + \CL) \oplus \ZZ/n\ZZ \oplus \left(\bigoplus_{i=1,2,3} \Jac(H_i)\right)$ is $2n$-equivalent to \[\ZZ/\frac{n}{\kappa}\ZZ \oplus \Jac(\hat{G})^2\oplus (\ZZ/2\ZZ)^{\tilde{t}+\epsilon-1}\oplus \left(\bigoplus_{\O}( \ZZ/k_\O\ZZ)\right) \oplus (\CP \cap \CL)/\CL' \oplus \Jac(G).\]  Using the first equation, this second term is equivalent to \[ \Jac(G) \oplus \Jac(\hat{G})^2\oplus \left(\bigoplus_{\O \text{ Type I or II}}( \ZZ/k_\O\ZZ)\right) \oplus \CD/(\CP + \CL)\]  \noindent from which the theorem follows. \end{proof}

\section{Examples and Applications}

\subsection{Klein Four Actions}

Let us begin by considering what Theorem \ref{T:Equiv} tells us about graphs admitting a Klein-Four action.  In particular, we note that the only orbits which contribute to the results are the orbits consisting of a single point fixed by the entire group action; all other orbits either are Type III or have $k=1$.

\begin{corollary} \label{C:Klein}
  Let $G$ be a graph which admits a $K$-action so that there are $\o\ge 1$ points that are fixed by the entire group $K$. Then for any prime $p \ne 2$, the $p$-part of $\Jac(G)\oplus \Jac(\hat{G})^2$ is isomorphic to the $p$-part of the direct sum $\Jac(H_1) \oplus \Jac(H_2) \oplus \Jac(H_3)$.  Moreover, if the $2$-part of $\Jac(H_1) \oplus \Jac(H_2) \oplus \Jac(H_3)$ is of order $2^n$ and the $2$-part of $\Jac(\hat{G})$ is of order $2^m$ then the $2$-part of $\Jac(G)$ has order $2^{n-\o-2m+1}$.
\end{corollary}

As an example, let $G$ be the graph corresponding to the vertices and edges of a regular octahedron, as pictured in Figure \ref{F:Octo}.  In particular, let $\{a,b,c,d\}$ be four vertices connected in a cycle, and let $x$ and $y$ be two additional vertices which are each connected to all of $\{a,b,c,d\}$ but not to each other.  We can define a Klein-Four action on this graph by letting $\sigma_1$ permute the vertices by $(a,b)(c,d)$ and $\sigma_2$ permute them as $(a,d)(b,c)$.  In particular, both $x$ and $y$ are fixed by the entire group action.

\begin{figure}[!htbp]
\centering
\subfloat[$W_8$]{\begin{tikzpicture}
  [scale=.8,auto=left,every node/.style={state,minimum size=.5cm,fill=blue!20}]
  \node (x) at (0,2) {x};
  \node (a) at (-1,1) {a};
  \node (b) at (1,1)  {b};
  \node (c) at (1,-1)  {c};
  \node (d) at (-1,-1) {d};
  \node (y) at (0,-2) {y};

\foreach \from/\to in {a/b,b/c,c/d,d/a,x/a,x/b,x/c,x/d,y/a,y/b,y/c,y/d}
    \draw (\from) -- (\to);

\end{tikzpicture}} \quad
\subfloat[$H_1 \cong H_2$]{
\begin{tikzpicture}
 [scale=.7,auto=left,every node/.style={state,minimum size=.5cm,fill=blue!20}]
  \node (0) at (0,2) {};
  \node (1) at (1,0) {};
  \node (2) at (-1,0)  {};
  \node (3) at (0,-2) {};
 \foreach \from/\to in {0/1,0/2,1/3,2/3,1/2}
    \draw (\from) -- (\to);
\end{tikzpicture}} \quad
\subfloat[$H_3$]{ \begin{tikzpicture}
  [scale=.7,auto=left,every node/.style={state,minimum size=.5cm,fill=blue!20}]
  \node (0) at (0,2) {};
  \node (1) at (1,0) {};
  \node (2) at (-1,0)  {};
  \node (3) at (0,-2) {};
 \foreach \from/\to in {0/1,0/2,1/3,2/3}
    \draw (\from) -- (\to);
    \draw (1) to[bend right=45] (2);
      \draw (1) to[bend left=45] (2);
\end{tikzpicture}}
\quad
\subfloat[$\hat{G}$]{\makebox[.1\textwidth]{\begin{tikzpicture}
  [scale=.7,auto=left,every node/.style={state,minimum size=.5cm,fill=blue!20}]
  \node (0) at (0,2) {};
  \node (1) at (0,0) {};
  \node (3) at (0,-2) {};
 \foreach \from/\to in {0/1,1/3}
    \draw (\from) -- (\to);
\end{tikzpicture}}}
\caption{The octahedron graph and its various quotients}
\label{F:Octo}
\end{figure}
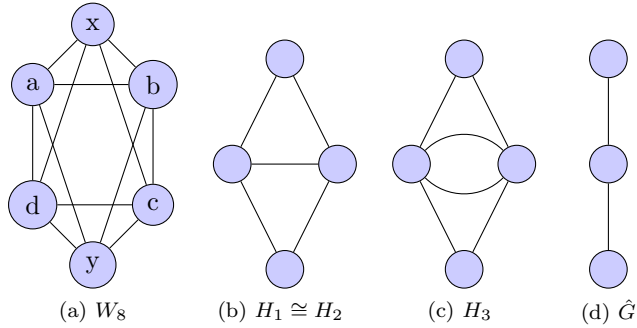

The quotients $H_1$ and $H_2$ are each isomorphic to two triangles attached at an edge, so we can easily compute that $\Jac(H_1) \cong \Jac(H_2) \cong \ZZ/8\ZZ$.  At the same time, the quotient by the involuton $\sigma_1\sigma_2$ will have critical group $\Jac(H_3) \cong \ZZ/12\ZZ$.  Because $\hat{G}=G/D_2$ is a tree, Corollary \ref{C:Klein} tells us that $\Jac(G) \cong \ZZ/3\ZZ \oplus K$, where $K$ is a finite abelian group of order $2^7$.  In fact, an explicit calculation shows that $\Jac(G) \cong \ZZ/3\ZZ \oplus \ZZ/2\ZZ \oplus (\ZZ/8\ZZ)^2$.  We note that this graph is also the circulant graph $C_6^{1,2}$ and as such it also admits an action by the group $D_3$; one can derive the same result by decomposing the graph according to this action and Theorem \ref{T:Equiv}.

\subsection{Graphs with large and small orbits}

In \cite{GM}, we considered the situation where all of the $D_n$-orbits had either $n$ or $2n$ points.  In the language of this paper, that meant that all orbits of Type I or II had index $1$ and orbits of Type III had index either $1$ or $2$.  In particular, Theorem \ref{T:Equiv} then implies that we have that $\Jac(H_1) \oplus \Jac(H_2) \oplus \Jac(H_3)\oplus \ZZ/n\ZZ$ is $2n$-equivalent to the group  $\Jac(G) \oplus \Jac(\hat{G})^2$.  This is equivalent to the conclusions of that paper.

At the other end of the spectrum, if we assume that $G$ is a graph that has a harmonic $D_n$-action such that a point $\omega$  is fixed by the entire action, then the orbit $\{\omega\}$ is a Type I orbit of index $n$.  In particular, the theorem then tells us that $\Jac(H_1) \oplus \Jac(H_2) \oplus \Jac(H_3)$ is $2n$-equivalent to the group  $\displaystyle \Jac(G) \oplus \Jac(\hat{G})^2\bigoplus_{\substack{\O \text{ Type I or II}\\ \O \ne \{\omega\}}}( \ZZ/k_\O\ZZ)$

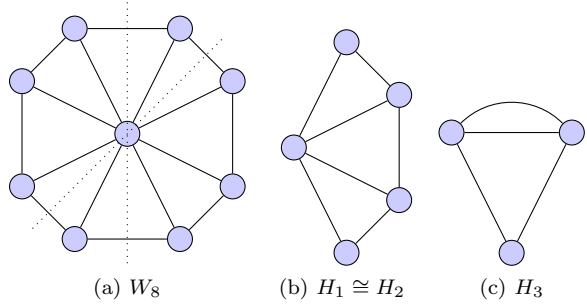
\begin{figure}[!htbp]
\centering
\subfloat[$W_8$]{\begin{tikzpicture}
  [scale=.7,auto=left,every node/.style={state,minimum size=.3cm,fill=blue!20}]
  \node (0) at (0,0) {};
  \node (1) at (1,-2) {};
  \node (2) at (-1,-2)  {};
  \node (3) at (-2,-1)  {};
  \node (4) at (-2,1) {};
  \node (5) at (-1,2) {};
  \node (6) at (1,2) {};
  \node (7) at (2,1) {};
  \node (8) at (2,-1) {};

\foreach \from/\to in {0/1,0/2,0/3,0/4,0/5,0/6,0/7,0/8,1/2,2/3,3/4,4/5,5/6,6/7,7/8,8/1}
    \draw (\from) -- (\to);
    \draw[dotted] (0,2.5) -- (0,-2.5) ;
    \draw[dotted] (-1.8,-1.8) -- (1.8,1.8);

\end{tikzpicture}} \quad
\subfloat[$H_1 \cong H_2$]{
\begin{tikzpicture}
 [scale=.7,auto=left,every node/.style={state,minimum size=.3cm,fill=blue!20}]
  \node (0) at (0,0) {};
  \node (1) at (1,-2) {};
  \node (6) at (1,2) {};
  \node (7) at (2,1) {};
  \node (8) at (2,-1) {};

\foreach \from/\to in {0/1,0/6,0/7,0/8,6/7,7/8,8/1}
    \draw (\from) -- (\to);
\end{tikzpicture}} \quad
\subfloat[$H_3$]{\begin{tikzpicture}
  [scale=.8,auto=left,every node/.style={state,minimum size=.3cm,fill=blue!20}]
  \node (0) at (0,0) {};
  \node (1) at (1,2) {};
  \node (2) at (-1,2)  {};
 \foreach \from/\to in {0/1,0/2,1/2}
    \draw (\from) -- (\to);
 \foreach \from/\to in {1/2}
    \draw (\from) to[bend right=45] (\to);
\end{tikzpicture}}
\caption{The graph $W_8$ and its various quotients}
\label{F:Wheel}
\end{figure}

As an example of a family of graphs with a fixed point, we consider the wheel graphs $W_{2n}$ obtained by starting with a $2n$-cycle and adding a single vertex in the center which is connected to each of the other vertices.  (We illustrate the specific case where $n=4$ in Figure \ref{F:Wheel}) This graph admits a harmonic $D_n$-action where $\sigma_1$ and $\sigma_2$ act by reflections through opposing pairs of edges.  One can see that the quotients $H_1=W_{2n}/\langle \sigma_1 \rangle$ and $H_2=W_{2n}/\langle \sigma_2 \rangle$ are each isomorphic to a `chain' of $n-1$ triangles, and it follows from results in \cite{BG} (or a direct calculation) that $\Jac(H_1) \cong \Jac(H_2) \cong  \ZZ/F_{2n-1}\ZZ$ where $F_k$ is the $k^{th}$ Fibonacci number, defined by the recurrence $F_0=F_1=1$ and $F_k=F_{k-1}+F_{k-2}$ for all $k \ge 2$. Moreover, one can see that $H_3=W_{2n}/\langle \sigma_1\sigma_2\rangle$ consists of a triangle with one edge doubled, and therefore $\Jac(H_3) \cong \ZZ/5\ZZ$.  Because the graph $\hat{G} = W_{2n}/D_n$ is a tree, our theorem therefore implies that $\Jac(W_{2n})$ is $2n$-equivalent to $(\ZZ/F_{2n-1}\ZZ)^2 \oplus \ZZ/5\ZZ$.  In many cases we have that $gcd(5F_{2n-1},2n)=1$, in which case this equivalence is actually an isomorphism.  We note that the case of a wheel graph was previously considered in \cite{Biggs} using different techniques.

\subsection{Square webs}
Consider the family of graphs $SW_n$ given by placing $n$ concentric squares on top of four radial lines, an example of which is  illustrated in Figure \ref{F:SW}.
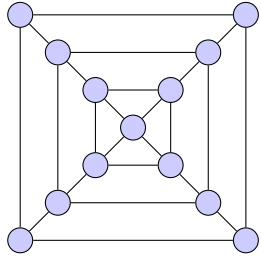
\begin{figure}[!htbp]
\begin{center}
\begin{tikzpicture}
  [scale=.5,auto=left,every node/.style={state,minimum size=0.3cm,fill=blue!20}]
  \node(a3) at (-3,-3) {};
  \node(a2) at (-2,-2){};
  \node(a1) at (-1,-1){};
  \node(b3) at (-3,3) {};
  \node(b2) at (-2,2){};
  \node(b1) at (-1,1){};
  \node(c3) at (3,3) {};
  \node(c2) at (2,2){};
  \node(c1) at (1,1){};
  \node(d3) at (3,-3) {};
  \node(d2) at (2,-2){};
  \node(d1) at (1,-1){};
  \node(x) at (0,0){};
  \foreach \from/\to in {a1/b1,b1/c1,c1/d1,d1/a1,a2/b2,b2/c2,c2/d2,d2/a2,a3/b3,b3/c3,c3/d3,d3/a3,a3/a2,a2/a1,a1/x,b3/b2,b2/b1,b1/x,c3/c2,c2/c1,c1/x,d3/d2,d2/d1,d1/x}
    \draw (\from) -- (\to);
\end{tikzpicture}
\caption{Square Web $SW_3$}
\label{F:SW}
\end{center}
\end{figure}

We will define a harmonic action of the group $(\ZZ/2\ZZ)^2$ on this graph by letting $\sigma_1$ be reflection in the $x$-axis and $\sigma_2$ be reflection in the $y$-axis.   This action will lead to a single fixed point and $n$ non-inertial orbits that have four vertices apiece.  Moreover, it is clear that $SW_n/\langle \sigma_1 \rangle \cong  SW_n/\langle \sigma_2 \rangle$ and $SW_n/(\ZZ/2\ZZ)^2$ is a tree, so corollary \ref{C:Klein} implies that $\Jac(SW_n)$ is $2$-equivalent to $(\Jac(SW_n/\langle \sigma_1 \rangle)^2 \oplus \Jac(SW_n/\langle \sigma_1\sigma_2 \rangle)$.   The graphs $SW_n/\langle \sigma_1 \rangle$ and $SW_n/\langle \sigma_1\sigma_2 \rangle$ can be viewed in Figure \ref{F:SWquot}.

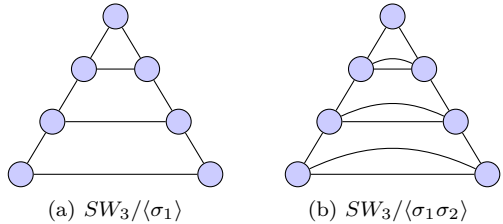
\begin{figure}[!htbp]
\centering
\subfloat[$SW_3/\langle \sigma_1 \rangle$]{\begin{tikzpicture}
  [scale=.7,auto=left,every node/.style={state,minimum size=.3cm,fill=blue!20}]
  \node (x) at (0,0) {};
  \node (a1) at (.6,-1) {};
  \node (a2) at (-.6,-1)  {};
  \node (b1) at (1.2,-2)  {};
  \node (b2) at (-1.2,-2) {};
  \node (c1) at (1.8,-3) {};
  \node (c2) at (-1.8,-3) {};

\foreach \from/\to in {x/a1,x/a2,a1/a2,a1/b1,b1/c1,a2/b2,b2/c2,b1/b2,c1/c2}
   \draw (\from) -- (\to);
\end{tikzpicture}}
\qquad \subfloat[$SW_3/\langle \sigma_1\sigma_2 \rangle$]{
\begin{tikzpicture}
   [scale=.7,auto=left,every node/.style={state,minimum size=.3cm,fill=blue!20}]
  \node (x) at (0,0) {};
  \node (a1) at (.6,-1) {};
  \node (a2) at (-.6,-1)  {};
  \node (b1) at (1.2,-2)  {};
  \node (b2) at (-1.2,-2) {};
  \node (c1) at (1.8,-3) {};
  \node (c2) at (-1.8,-3) {};

\foreach \from/\to in {x/a1,x/a2,a1/a2,a1/b1,b1/c1,a2/b2,b2/c2,b1/b2,c1/c2}
   \draw (\from) -- (\to);

    \foreach \from/\to in {a1/a2,b1/b2,c1/c2}
    \draw (\from) to[bend right=25] (\to);

\end{tikzpicture}}
\caption{The quotient graphs of $SW_3$}
\label{F:SWquot}
\end{figure}

Each of these quotient graphs can be thought of as a chain of polynomials in the sense of \cite{BG}, and in particular that paper shows that each of these critical groups will be cyclic.  More precisely, Theorem 3.1 of that paper implies that $\Jac(SW_n/\langle \sigma_1 \rangle) \cong \ZZ/a_n\ZZ$ and $\Jac(SW_n/\langle \sigma_1\sigma_2 \rangle) \cong \ZZ/b_n\ZZ$, where $a_n$ satisfies the recurrence relation $a_0=1, a_1=3, a_n=4a_{n-1}-a_{n-2}$ for $n \ge 2$ and $b_n$ satisfies the recurrence relations $b_0=1, b_1=5, b_{n}=6b_{n-1}-b_{n-2}$ for $n \ge 2$.  One can easily deduce that all of the $a_n$ and $b_n$ are odd, and we thus obtain that  $\Jac(SW_n) \cong (\ZZ/a_n\ZZ)^2 \oplus \Jac(\ZZ/b_n\ZZ)$.  We note that $a_n$ (resp. $b_n$) is sequence $A001835$ (resp. $A001653$) in the Online Encyclopedia of Integer Sequences\cite{OEIS}, and explicit formulas are given by
\[a_n=\frac{(1+\sqrt{3})(2+\sqrt{3})^n+(1-\sqrt{3})(2-\sqrt{3})^n}{2}\] \[ b_n=\frac{(\sqrt{2}+1)(3+2\sqrt{2})^n+(\sqrt{2}-1)(3-\sqrt{2})^n}{2\sqrt{2}}\]

\subsection{Example where groups are isomorphic}

Let $G$ be the graph in Figure \ref{F:counter}, and let $\sigma_1$ act on the vertices by the permutation $( x \, y )(a \, b)$  and $\sigma_2$ act by the permutation  $( x \, y )(a \, c)$.  Then $\langle \sigma_1,\sigma_2 \rangle \cong D_3$ and $\{x,y\}$ is an orbit of Type III and index $3$ while $\{a,b,c\}$ is an orbit of Type I and index $1$.  One can compute that $\Jac(G/\langle\sigma_1\rangle) \cong \Jac(G/\langle\sigma_2\rangle) \cong \ZZ/2\ZZ$ and see that $G/\langle\sigma_1\sigma_2\rangle$ is a tree so has trivial Jacobian, which also implies that $\Jac(\hat{G})$ is trivial.  Moreover, it is a straightforward calculation to see that $\Jac(G) \cong \ZZ/2\ZZ \oplus \ZZ/6\ZZ$.

\begin{figure}[!htbp]
\begin{center}
\begin{tikzpicture}
  [scale=1.1,auto=left,every node/.style={state,minimum size=0.6cm,fill=blue!20}]
  \node(x) at (-2,0) {$x$};
  \node(y) at (2,0){$y$};
  \node(a) at (0,1){$a$};
  \node(b) at (0,0){$b$};
  \node(c) at (0,-1){$c$};
  \foreach \from/\to in {a/x,a/y,b/x,b/y,c/x,c/y}
    \draw (\from) -- (\to);
\end{tikzpicture}
\caption{Graph with $D_3$ action}
\label{F:counter}
\end{center}
\end{figure}
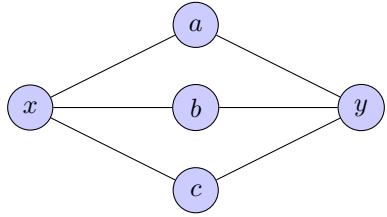

In particular, $\displaystyle \Jac(G) \oplus \Jac(\hat{G})^2\oplus \left(\bigoplus_{\O \text{ Type I or II}}( \ZZ/k_\O\ZZ)\right) \cong \ZZ/2\ZZ \oplus \ZZ/6\ZZ$ and $\Jac(H_1) \oplus \Jac(H_2) \oplus \Jac(H_3) \oplus (\ZZ/n\ZZ) \cong (\ZZ/2\ZZ)^2 \oplus (\ZZ/3\ZZ)$.  As the theorem predicts, these two groups are $6$-equivalent.  In fact, in this case they are isomorphic.

\vskip .2in

\noindent {\bf Acknowledgments}

The author would like to thank Criel Merino for helpful conversations and ideas throughout this project, and the anonymous referees for suggestions that improved the exposition.

\end{document}